\documentclass[11pt]{amsart}	
\usepackage[utf8]{inputenc}
\usepackage{graphicx}
\usepackage{subcaption}
\usepackage[usenames,dvipsnames]{pstricks}
\usepackage{epsfig}
\usepackage{pst-grad} 
\usepackage{pst-plot} 
\usepackage[space]{grffile} 
\usepackage{etoolbox} 
\makeatletter 
\patchcmd\Gread@eps{\@inputcheck#1 }{\@inputcheck"#1"\relax}{}{}
\makeatother
\usepackage{comment}
\usepackage{mathtools}
\usepackage{amsmath}
\usepackage{amssymb}
\usepackage{amsthm}
\usepackage{ifthen}
\usepackage{color}
\usepackage[UKenglish]{babel}
\usepackage{polski}
\usepackage{hyperref}
\usepackage[shortlabels]{enumitem}
\usepackage{esint}
\usepackage{soul}
\numberwithin{equation}{section}

\newcommand{\intav}[1]{\mathchoice {\mathop{\vrule width 6pt height 3 pt depth  -2.5pt
\kern -8pt \intop}\nolimits_{\kern -6pt#1}} {\mathop{\vrule width
5pt height 3  pt depth -2.6pt \kern -6pt \intop}\nolimits_{#1}}
{\mathop{\vrule width 5pt height 3 pt depth -2.6pt \kern -6pt
\intop}\nolimits_{#1}} {\mathop{\vrule width 5pt height 3 pt depth
-2.6pt \kern -6pt \intop}\nolimits_{#1}}}

\def\polhk#1{\setbox0=\hbox{#1}{\ooalign{\hidewidth\lower1.5ex\hbox{`}\hidewidth\crcr\unhbox0}}}

\newcommand{\supp}{\operatorname{supp}}

\renewcommand{\div}{\operatorname{div}}

 \newcommand{\Rr}{\mathbb R}

\newcommand{\Nn}{\mathbb N}

\renewcommand{\div}{\operatorname{div}}

\newcommand{\dist}{\operatorname{dist}}

\DeclareMathOperator*{\esslimsup}{ess\,lim\,sup}

\newtheorem{Theorem}{Theorem}
\numberwithin{Theorem}{section} 
\newtheorem{Definition}{Definition}
\numberwithin{Definition}{section} 
\newtheorem{Lemma}{Lemma}
\numberwithin{Lemma}{section} 
\newtheorem{Corollary}{Corollary}
\numberwithin{Corollary}{section} 
\newtheorem{Proposition}{Proposition}
\numberwithin{Proposition}{section} 

\newtheorem{assump}{}

\newtheoremstyle{break}
  {\topsep}{\topsep}%
  {\itshape}{}%
  {\bfseries}{}%
  {\newline}{}%
\theoremstyle{break}

\theoremstyle{definition}

\theoremstyle{remark}
\newtheorem{Remark}{Remark}
\newcommand{\p}{\partial}

\usepackage{setspace}
\title{A fully nonlinear transmission problem degenerating on the interface}

\author[D. Giovagnoli and D. Jesus]{Davide Giovagnoli \and David Jesus}
\address{David Jesus, Università di Bologna, Bologna, Piazza di Porta San Donato 5, 40126, Italy}{}
\email{david.jesus2@unibo.it}
\address{Davide Giovagnoli, Università di Bologna, Bologna, Piazza di Porta San Donato 5, 40126, Italy}{}
\email{d.giovagnoli@unibo.it}

\begin{document}

\maketitle
{\sc Abstract.} In this paper we prove that solutions to a transmission problem degenerating on the interface are H\"older differentiable up to the interface with universal estimates. Furthermore, we obtain a sharper pointwise $C^{1,\alpha(\cdot)}$ with  optimal variable exponent and uniform estimates.
\footnote{\small  This research is partially supported by PRIN 2022 7HX33Z - CUP J53D23003610006. 
}

\section{Introduction}

We are concerned with regularity of viscosity solutions to the following degenerate fixed transmission problem
\begin{align}\label{eq:main}
    \begin{cases}
        |x_d-\Psi(x')|^{a(x)} F^\pm(D^2u )=f^\pm,  \quad &\mbox{ in } \Omega^\pm:=B_1\cap\{\pm (x_d-\Psi(x'))>0\}\\
        u^+_\nu-u^-_\nu=g\, &\mbox{ on } \Gamma:=B_1\cap \{x_d=\Psi(x')\}
    \end{cases}
\end{align}
where $0 \leq a < 1$ and  $\nu$ is the normal of $\Gamma$ pointing to $\Omega^+$. Note that the operator degenerates as a distance to the interface.

Transmission problems model physical phenomena in which the behavior changes across some fixed interface and have attracted considerable attention throughout the years, starting with the pioneering work of Picone \cite{Picone} in elasticity in the 1950s and subsequent works \cite{Campanato,Lions,Schechter,Stampacchia}. For a comprehensive study of these problems see \cite{Borsuk}. For other recent developments, see \cite{CSCS,Citti-Ferrari,DFS2014,DFS2015A,DFS2015B,DFS2018,DFS2019,Kriventsov,Li-Nirenberg,Li-Vogelius} and references therein.

On the other hand, equations which degenerate as a distance to a hypersurface generalize the well-known Muckenhoupt weights, which have many important applications in harmonic analysis, partial differential equations, and related areas. These equations have been extensively considered, see for example \cite{AFV, Caffarelli-Silvestre_2007, Dong-Kim_2015, Fabes-Jerison-Kenig_1982, Fabes-Kenig-Serapioni_1982, Krylov_1994, Krylov_1999, Sire-Terracini-Vita_2020, Sire-Terracini-Vita_2021}. Of particular interest is the equation in divergence form
\begin{align}\label{eq_VS}
    \mathcal{L}_au:=\div\left(|y|^aA(x,y) Du\right)=f
\end{align}
which has a close relation to the fractional Laplacian $(-\Delta)^s$ for $0<s<1$ by the famous paper of Caffarelli and Silvestre \cite{Caffarelli-Silvestre_2007}. The nondivergent case was considered in the recent paper \cite{JS} by the second author and Yannick Sire. 

We would like to remark that elliptic and parabolic equations
involving coefficients degenerating as the distance to some submanifold appear
very naturally in the analysis of singular manifolds with conic or conic-edge
singularities (see for instance \cite{haoFang1, haoFang2}). In this situation, the degeneracy exponent depends explicitly on the sharpness of the edge angle, which can be variable. When we study $C$ viscosity solutions of such problems, however, we note that uniqueness of solutions fails drastically, since the hypersurface where the ellipticity degenerates creates a natural interface, which disconnects the domain into two components. The strategy developed in \cite{JS} to circumvent this problem was to instead consider $L^p$ viscosity solutions which do not see sets of zero measure. Thus, in order to reinstate uniqueness in the context of $C$ viscosity solutions, it is natural to impose a transmission condition across the interface.

The purpose of this paper is to improve the results in \cite{SS}, by considering operators which degenerate as we approach the interface with a variable exponent, and we obtain sharp pointwise regularity which depends on the pointwise value of this exponent. Since the estimates are uniform with respect to the point, as a direct consequence of our theorem, we additionally get local $C^{1,\alpha}$ regularity up to the interface with uniform estimates. Furthermore we provide a more direct proof of the $C^{1,\alpha}$ regularity.

\section{Main result}

Here and in the rest of the paper, the constant $\alpha_0(\lambda,\Lambda,d)$ always refers to the universal exponent corresponding to the interior $C^{1,\alpha_0}(B_{1/2})$ regularity of solutions to the problem $F(D^2u)=0$ in $B_1$ where $F$ is any $(\lambda,\Lambda)$ uniformly elliptic operator, see \cite[Corollary 5.7]{CC}. 
Henceforth $\omega(x)$ denotes the function  $\omega(x)=  |x_d-\Psi(x')|^{a(x)}$.

Throughout the paper we shall always work under the following assumptions.

\begin{assump}[Uniform Ellipticity]\label{Assumption1}
There exist constants $0<\lambda\leq \Lambda<\infty$  such that for every matrices $M, N\in S(d)$ with $N\geq 0$, 
\[
\lambda\, \,|N|\leq F^\pm(M+N)-F^\pm(M)\leq \Lambda\, \, |N|.
\]
\end{assump}
For future reference, we  define the modulus of continuity $\gamma_h$ of a function $h$ by
\[
\gamma_h(t)=\sup_{|x-y|\leq t}|h(x)-h(y)|, \quad \mbox{ for } t>0.
\]
\begin{assump}[Continuity and integrability]\label{Assumption2}
    Assume that  $f^\pm\in C(\Omega^\pm)\cap L^\infty(\Omega^\pm)\cap L^p_\omega(\Omega^\pm)$ with $ p=d(1+1/\bar a)/2>d $, $g\in C_c^{\alpha_0}(\Gamma)$, $\Psi\in C^{1,\alpha_0}(\overline{B_1'})$, and $a$  has modulus of continuity satisfying
\[
\limsup_{t\to 0}\ln\left(\frac{1}{t}\right)\gamma_{a}(t)=0,
\]
and $\bar a:= \max_{B_1} a(x) < 1$. 
With $f^\pm \in L_\omega^p(\Omega^\pm)$ we mean that $f^\pm \omega^{-1} \in L^p(\Omega^\pm)$.
\end{assump}

\begin{assump}[Behavior near $\Gamma$]\label{Assumption3}
 There exists $\theta$ small depending only on $d$, $\lambda$, $\Lambda$ such that
\[
\sup_{M\in S(d)\setminus \{0\}}\frac{|F^+(M)-F^-(M)|}{|M|}\leq \theta.
\]
\end{assump}

We obtain two results. The first one is a uniform local regularity result, which is actually a direct consequence of the second result. We include it because it is stated in a more conventional way.

\begin{Theorem} \label{Thm:local} Suppose that the assumptions \ref{Assumption1}-\ref{Assumption3} are in force and let $u \in C(B_1)$  be a viscosity solution to the transmission problem \eqref{eq:main}. Then $u \in C^{1, \alpha}(\overline{\Omega_{1/2}^{\pm}})$ with $\alpha = \min \{\alpha_0^-, 1- \bar a\}$ where $\Omega_{1/2}^\pm := \Omega^\pm \cap B_{1/2}$. Moreover 
it is endowed with the local estimate up to the interface
\begin{equation*}
\displaystyle
    \| u\|_{C^{1,\alpha} (\overline{\Omega_{1/2}^{\pm}})} \leq C \| \Psi \|_{C^{1,\alpha_0}(\overline{B'_1})} \left(  \| u\|_{L^\infty(B_1)} + \|g \|_{C^{\alpha_0}(\Gamma)} + \|f^+\|_{L^p_\omega(\Omega^+)}+\|f^-\|_{L^p_\omega(\Omega^-)} \right)
\end{equation*}
where $ p=d(1+1/\bar a)/2>d $ and  $C$ universal depending only on $d, \lambda, \Lambda, \bar a$ and $\alpha_0$.
\end{Theorem}

The next theorem constitutes the main result of this paper.

\begin{Theorem} \label{Thm:main} Suppose that the assumptions \ref{Assumption1}-\ref{Assumption3} are in force and let $u \in C(B_1)$  be a viscosity solution to the transmission problem \eqref{eq:main}. 
Then, for every $x_0\in \Omega_{1/2}^{\pm}$, $u \in C^{1, \alpha(x_0)}(x_0)$ with $\alpha(x_0) = \min \{\alpha_0^-, 1- a(x_0)\}$. Moreover 
it is endowed with the local estimate up to the interface
\begin{equation}
\displaystyle
    \| u\|_{C^{1,\alpha(\cdot)} \left(\overline{\Omega_{1/2}^{\pm}}\right)} \leq C \| \Psi \|_{C^{1,\alpha_0}(\overline{B'_1})} \left(  \| u\|_{L^\infty(B_1)} + \|g \|_{C^{\alpha_0}(\Gamma)} + \|f^+\|_{L^p_\omega(\Omega^+)}+\|f^-\|_{L^p_\omega(\Omega^-)} \right)
\end{equation}
where $ p=d(1+1/\bar a)/2>d $ and  $C$ universal depending only on $d, \lambda, \Lambda, \bar a$ and $\alpha_0$.
\end{Theorem}

Note that as a simple consequence of this result, solutions of \eqref{eq:main} are Lipschitz continuous across the interface. This the best we can expect for this problem with $g \not\equiv 0$, since it prescribes a jump of the gradient.

The H\"older spaces with variable exponent are defined in Section \ref{Section:preliminaries}. The expression $\alpha=\min\{\alpha_0^-,1-a\}$ should be understood in the following sense: if $\alpha_0>1-a$, then solutions are $C^{1,\alpha}$ for $\alpha=1-a$. On the other hand, if $\alpha_0\leq 1-a$, then solutions are $C^{1,\alpha}$ for every $\alpha<\alpha_0$. 

This regularity result is optimal for this set of assumptions, as illustrated by the counterexamples constructed in \cite{JS}.

Now we comment on the assumptions \ref{Assumption1}-\ref{Assumption3}. 
From \ref{Assumption1} the operators $\omega(x)F^\pm$ are degenerate elliptic according to the definition given in \cite{JS} respectively in $\Omega^\pm$; \ref{Assumption2} settles the  summability assumption on $f^\pm$ necessary for boundary estimates as in \cite{BGMW}. The choice of $\alpha_0$ as the exponent for the H\"older continuity of the data is done for simplicity and without loss of generality. Moreover, the modulus of continuity of $a$ ensures that we get sharp pointwise regularity, see \cite{J}; the $\theta$ in \ref{Assumption3} is the same as in \cite{SS} and is used to obtain regularity at points in $\Gamma$ where $g=0$.

\subsection{Strategy of the results.}
The approach of the paper combines the arguments in \cite{JS} about fully nonlinear equations with degeneracy of distance-type, with the treatment of fully nonlinear transmission problem due to \cite{SS}, together with some ideas from \cite{J} to obtain sharper results. The presence of a degeneracy in the interface poses new difficulties and requires new arguments that will considerably differ from \cite{SS}, particularly when constructing barriers. The first part of the work focuses on developing an ABP, which is then used to obtain a Harnack inequality and H\"older regularity for solutions of \eqref{eq:main}, and is the subject of Section \ref{Section:ABP} and Section \ref{Section:Harnack}.

Then we include a comprehensive theory for the transmission problem with the flat interface and constant exponents. For this problem we provide in Section \ref{Section:flat} uniqueness and existence. 
 
In Section \ref{Section:approximation} we prove a stability result, which roughly states that as we zoom onto the interface, the problem approximates the flat one. Although this result was not needed to prove our main results, we decided to include it because stability results are interesting on their own. This constitutes maybe the most challenging part of the paper, since the usual barriers don't work for the degenerate case. Therefore, we have to build new $W^{2,p}$ barriers whose Hessian blows up near the interface, which compensates the degenerate behavior of the equation. This proof substantially differs from the literature. 

The approach we have pursued up to this point is very flexible, as it also allows for the consideration of general operators with variable coefficients of the form $F^\pm(M,x)$ with the assumptions as in \cite{JS}. However, in Lemma \ref{Lm:approxalpha0} as well as Proposition \ref{Pro:geom_ite}, it seems crucial to prove that the difference of solutions belongs to the $S$ class, which is a very delicate and difficult issue in the theory of viscosity solutions (as is discussed for instance in \cite[Section 1.1]{SS1}), therefore we were forced to drop the full generality.

Section \ref{Section:approxalpha0} consists of the approximation lemma which relates our equation with the limiting profile where both $f^\pm$ and $g$ are zero. This proof is very similar to \cite{SS}.

The purpose of Section \ref{Section:Gradient_reg} is to adapt the argument developed by Caffarelli in \cite{C} which consists of importing the improved regularity for the limiting profile (corresponding to the case $g=0$ and $f^\pm=0$) to our equation. This is done via geometric iterations, where we subsequently approximate our solution by affine functions and rescale, zooming into a point in the interface. To get a two-sided improvement up to the interface, we need to consider different affine functions from each side. By writing explicitly how these affine functions depend on the transmission condition $g$, we are able to maintain the smallness assumption on $\|g\|_{C^{\alpha_0}}$ which considerably simplifies the proof, since the case when $g\approx0$ is much simpler than the case where $g$ is large. Furthermore, the proper rescaling is dependent on the degeneracy, which has variable exponent. Thus to obtain the sharp regularity, a careful application of the argument developed in \cite{J} has to be performed, by considering a different rescaling power in each iteration. Once we have pointwise regularity at the interface, we can patch it with the classical interior regularity in the usual way.

\section{Preliminaries} \label{Section:preliminaries}
We start with the definition of viscosity solutions to \eqref{eq:main}.
\begin{Definition}[Viscosity solution]
    We say $u\in USC(B_1)$ is a viscosity subsolution of \eqref{eq:main} if for any $\varphi$ touching $u$ from above at $x_0\in B_1$, the following hold:
    \begin{enumerate}
        \item  if $x_0\in \Omega^\pm$ and $\varphi\in C^2(B_\delta(x_0)\cap \Omega^\pm)$ then
        \begin{align*}
           |(x_0)_d-\Psi(x'_0)|^{a(x_0)}  F^\pm(D^2\varphi(x_0))\geq f^\pm(x_0);
        \end{align*}
        \item if $x_0\in \Gamma$ and $\varphi\in C^1(B_\delta(x_0)\cap \overline{\Omega^-})\cap C^1(B_\delta(x_0)\cap \overline{\Omega^+})$ then
        \begin{align*}
            \varphi^+_\nu(x_0)-\varphi^-_\nu(x_0)\geq g(x_0).
        \end{align*}
    \end{enumerate}
    We say $u\in LSC(B_1)$ is a viscosity supersolution of \eqref{eq:main} if for any $\varphi$ touching $u$ from below at $x_0\in B_1$, the following hold:
    \begin{enumerate}
        \item  if $x_0\in \Omega^\pm$ and $\varphi\in C^2(B_\delta(x_0)\cap \Omega^\pm)$ then
        \begin{align*}
            |(x_0)_d-\Psi(x'_0)|^{a(x_0)}F^\pm(D^2\varphi(x_0))\leq f^\pm(x_0);
        \end{align*}
        \item if $x_0\in \Gamma$ and $\varphi\in C^1(B_\delta(x_0)\cap \overline{\Omega^-})\cap C^1(B_\delta(x_0)\cap \overline{\Omega^+})$ then
        \begin{align*}
            \varphi^+_\nu(x_0)-\varphi^-_\nu(x_0)\leq g(x_0).
        \end{align*}
    \end{enumerate}
    We say $u\in C(B_1)$ is a viscosity solution of \eqref{eq:main} if it is both a subsolution and a supersolution.
\end{Definition}

We say $u\in \underline{S}_{\, \omega}(f^\pm)$ if
\begin{align*}
    \omega(x)\mathcal{M}^+(D^2u)\geq f^\pm, \mbox{ in } \Omega^\pm.
\end{align*}
We define $\overline{S}_\omega(f^\pm)$, $S^*_\omega(f^\pm)$ and $S_\omega(f^\pm)$ in the obvious way, following the notation in \cite{CC}.

Next we define the H\"older spaces with variable exponent and present a more general characterization of these spaces, obtained in \cite[Proposition 1]{J}.
\begin{Definition}\label{Def:Holder_variable}
    We say $u\in C^{1,\alpha(\cdot)}(\Omega)$ if $u\in C^1(\Omega)$ and
    \begin{align*}
        [u]_{C^{1,\alpha(\cdot)}(\Omega)}:=\sup_{x\neq y\in \Omega}\frac{|Du(x)-Du(y)|}{|x-y|^{\alpha(x)}}<\infty.
    \end{align*}
    We define the norm $\|u\|_{C^{1,\alpha(\cdot)}(\Omega)}=\|u\|_{L^\infty(\Omega)} + \|Du\|_{L^\infty(\Omega)} + [u]_{C^{1,\alpha(\cdot)}(\Omega)}$.
\end{Definition}

\begin{Proposition}\label{Pro:new_holder}
Suppose we can find $r<1$ and sequences of affine functions $\ell_k(x)=a_k+b_k\cdot x $ and exponents $\alpha_k\uparrow \alpha(x_0)$, such that $(\alpha_k-\alpha)=o(k)$ and
\[
\left\|u-\ell_k\right\|_{L^\infty(B_{r^k}(x_0))}\leq K r^{k(1+\alpha_k)}.
\]
Then $u\in C^{1,\alpha(x_0)}(x_0)$ with constant $C(r)K$ and $0<\alpha(x_0)<1$.
\end{Proposition}

The equivalence between Definition \ref{Def:Holder_variable} and Proposition \ref{Pro:new_holder} is given by the following proposition.

\begin{Proposition}
Let $u\in C^1(\Omega)$ and $\alpha(\cdot)$ be uniformly continuous with modulus of continuity satisfying
$$ \lim_{t\to 0^+}\ln (t) \,\gamma_\alpha(t)=0. $$
Then $u\in C^{1,\alpha(\cdot)}(B_1)$ if and only if for every $x\in B_1$, there exists an affine function $l_x$ such that for every $r>0$ it holds
\begin{align}\label{eq_variable_approximation}
    \|u-l_x\|_{L^\infty(B_r(x))}\leq Cr^{1+\alpha(x)},
\end{align}
where $C>0$ is independent of $x$.
\end{Proposition}
\begin{proof}
We start by proving the first implication. Suppose 
\begin{align}\label{eq1}
    \sup_{x\neq y\in \Omega}\frac{|Du(x)-Du(y)|}{|x-y|^{\alpha(x)}}\leq C.
\end{align}
Let $x,y\in \Omega$, $x\neq y$ be arbitrary and take $r=|x-y|$. Then
$$ |Du(x)-Du(y)|\leq Cr^{1+\alpha(y)}. $$
Define $l_y(x)=u(y)+Du(y)\cdot (x-y)$. We want to prove \eqref{eq_variable_approximation}. We proceed by contradiction: suppose we can find a sequence of radii $r_k\to0$ such that
\begin{align}\label{eq2}
    \|u-l_y\|_{L^\infty(B_{r_k}(y))}r_k^{-(1+\alpha(y))}\geq k,
\end{align}
that is, for every $x\in B_{r_k}(y)$ it holds
$$ |u(x)-u(y)-Du(y)\cdot(x-y)|r_k^{-(1+\alpha(y))}\geq k $$
There exists $\eta$ in the line segment between $x$ and $y$ such that
$$u(x)-u(y)=Du(\eta)\cdot(x-y).$$
Thus
$$ |u(x)-u(y)-Du(y)\cdot(x-y)|\leq |Du(\eta)-Du(y)|r_k \leq Cr_k^{1+\alpha(y)},$$
which immediately produces a contradiction with \eqref{eq2}.

For the reverse implication, we start by writing
\begin{align*}
    &u(x)=u(y)+Du(y)\cdot(x-y)+O(r^{1+\alpha(y)})\\
    &u(y)=u(x)+Du(x)\cdot(y-x)+O(r^{1+\alpha(x)}),
\end{align*}
thus we can write
$$ (Du(x)-Du(y))\cdot(x-y)= O(r^{1+\alpha(x)})+O(r^{1+\alpha(y)}).$$
It suffices to prove $O(r^{1+\alpha(x)})=O(r^{1+\alpha(y)})$. For this purpose, we use the modulus of continuity $\gamma_\alpha$. Indeed, we have
$$ e^{\ln(r)\gamma_\alpha(r)}\leq r^{\alpha(x)-\alpha(y)}\leq e^{-\ln(r)\gamma_\alpha(r)}, $$
where $e^{\pm\ln(r)\gamma_\alpha(r)}=1+O(1)$. Thus
$$ O(r^{1+\alpha(x)})=O(r^{1+\alpha(y)})\left(1+O(1)\right)=O(r^{1+\alpha(y)}), $$
which completes the proof.
\end{proof}

\begin{Remark}\label{Rmk:mod_cont_exp}
    Note that assumption \ref{Assumption2} implies that, for every $0<r<1/e$, it holds
    \[
    \limsup_{k\to +\infty} k\gamma_{a}(r^kx)=0.
    \]
    Therefore, for every $\varepsilon_1>0$ there exists $\delta_1>0$ such that if $\rho<\delta_1$, then for every $k\in \Nn$,
    \[
     k\ln (1/\rho)\gamma_{a}(\rho^kx)\leq \varepsilon_1.
    \]
    Since $\rho$ will be chosen very small, we can assume that $\rho<1/e$ and thus
    \[
     k\gamma_{a}(\rho^kx)\leq \varepsilon_1.
    \]
    Now we fix 
    \[
    \varepsilon_1=\frac{\alpha_0-\alpha}{2}
    \]
    where $\alpha$ will be chosen later. This also fixes $\delta_1$, which now depends only on $\alpha$, universal constants and the data.
\end{Remark}

\section{Aleksandrov-Bakelman-Pucci estimate} \label{Section:ABP}

In this section we aim to prove the following ABP estimate.
\begin{Theorem}[ABP]\label{Thm:deg_ABP}
Let u satisfy
\begin{align*}
    \begin{cases}
    u\in \overline{S}_\omega(f^\pm), \quad &\mbox{ in } \Omega^\pm,\\
    u^+_\nu-u^-_\nu\leq g, &\mbox{ on } \Gamma,
\end{cases}
\end{align*}
with $f^\pm\in C(\Omega^\pm)\cap L^\infty(B_1)$, $g\in L^\infty(\Gamma)$ and $\Psi\in C^{1, \alpha_0}(\overline{B}_1')$. Then
\begin{align*}
    \sup_{B_1}u_-\leq \sup_{\partial B_1}u_-+C\left(\max_\Gamma g_+ +\|f^-_+\|_{L^d_\omega(\Omega^-)}+\|f^+_+\|_{L^d_\omega(\Omega^+)}\right).
\end{align*}

\end{Theorem}

We begin by stating the following existence result, which follows immediately from \cite[Lemma 3.1]{CCKS}. 

\begin{Lemma}\label{Lm:exist_CCKS}
Let $h^\pm\in L^p(\Omega^\pm)$ with $p>d$. There exists an $L^p$-strong solution $u\in W^{2,p}(B_1)$ of
\begin{align*}
    \begin{cases}
        \mathcal{M}^+(D^2u)\leq h^\pm,\quad &\mbox{ in } \Omega^\pm,\\
        u^+_\nu-u^-_\nu=0,&\mbox{ on } \Gamma,\\
        u=0, &\mbox{ on } \partial B_1.
    \end{cases}
\end{align*}
Moreover, $u$ satisfies
\begin{align*}
    \|u\|_{W^{2,p}(B_{1/2})}\leq C\left(\|h^-\|_{L^p(B_1)}+\|h^+\|_{L^p(B_1)}\right)
\end{align*}
and
\begin{align*}
    \|u\|_{L^\infty(B_1)}\leq C\left(\|h^-\|_{L^p(B_1)}+\|h^+\|_{L^p(B_1)}\right),
\end{align*}
where $C=C(p,d,\lambda,\Lambda)$.
\end{Lemma}

\begin{proof}
Let $h=h^+\chi_{\Omega^+}+h^-\chi_{\Omega^-}\in L^p$. By \cite[Lemma 3.1]{CCKS}, there exists an $L^p$-strong solution $u\in W^{2,p}(B_1)$ of
\begin{align*}
    \begin{cases}
        \mathcal{M}^+(D^2u)\leq h,\quad &\mbox{ in } B_1,\\
        u=0, &\mbox{ on } \partial B_1,
    \end{cases}
\end{align*}
which satisfies the desired estimates. Since $p>d$, by the Sobolev embedding, $W^{2,p}(B_1)\subset C^{1,1-d/p}(\overline{B_1})$  and thus the transmission condition is satisfied in the classical way.
\end{proof}

We will also use the ABP obtained in \cite{SS} for the uniformly elliptic case $\omega\equiv 1$.
\begin{Theorem}\label{Thm:unifelip_ABP}
    Let u satisfy
\begin{align*}
    \begin{cases}
    u\in \overline{S}_1(f^\pm), \quad &\mbox{ in } \Omega^\pm,\\
    u^+_\nu-u^-_\nu\leq g, &\mbox{ on } \Gamma,
\end{cases}
\end{align*}
with $f^\pm\in C(\Omega^\pm)\cap L^\infty(B_1)$, $g\in L^\infty(\Gamma)$ and $\Psi\in C^{1, \alpha_0}(\overline{B}_1')$. Then
\begin{align*}
    \sup_{B_1}u_-\leq \sup_{\partial B_1}u_-+C\left(\max_\Gamma g_+ +\|f^-_+\|_{L^d(\Omega^-)}+\|f^+_+\|_{L^d(\Omega^+)}\right).
\end{align*}
\end{Theorem}

We are now ready to prove the ABP in the degenerate setting.

\begin{proof}[Proof of Theorem \ref{Thm:deg_ABP}]
    We follow the reasoning in \cite{CCKS}. 

Let $\eta_j\in C^\infty(B_1)$, $0\leq \eta_j\leq j$, $\eta_j\to \omega^{-1}$ in $L^p$.  By Lemma \ref{Lm:exist_CCKS}, there exist functions $\phi_j\in W^{2,p}(B_1)\cap C(\overline{B_1})$ which solve
\begin{align*}
    \begin{cases}
        \mathcal{M}^+(D^2\phi_j)\leq f^\pm(\eta_j-\omega^{-1}),\quad &\mbox{ in } \Omega^\pm,\\
        (\phi_j)_\nu^+-(\phi_j)_\nu^-=0, &\mbox{ on } \Gamma,\\
        \phi_j=0,& \mbox{ on } \partial B_1,
    \end{cases}
\end{align*}
and satisfy
\begin{align*}
    \|\phi_j\|_{L^\infty(B_1)}\leq C\|f(\eta_j-\omega^{-1})\|_{L^p(B_1)}\to0, \mbox{ as } j\to \infty.
\end{align*}
Note that
\begin{align*}
    \mathcal{M}^-(D^2u+D^2\phi_j)\leq \mathcal{M}^-(D^2u)+\mathcal{M}^+(D^2\phi_j).
\end{align*}
Set $v=u+\phi_j$ to get
\begin{align*}
    \mathcal{M}^-(D^2v)\leq f^\pm\eta_j,\quad \mbox{ in } \Omega^\pm.
\end{align*}
Furthermore,
\begin{align*}
    v^+_\nu-v^-_\nu\leq g, \mbox{ on } \Gamma.
\end{align*}

Since the source term is continuous in $\Omega^\pm$, we can apply Theorem \ref{Thm:unifelip_ABP} to $v$ and get
\begin{align*}
    \sup_{B_1}v_-\leq \sup_{\partial B_1}v_-+C\left(\max_\Gamma g_++\|f^-_+\eta_j\|_{L^d(\Omega^-)}+\|f^+_+\eta_j\|_{L^d(\Omega^+)}\right).
\end{align*}
Letting $j\to \infty$, we get the desired result.
    
\end{proof}

\begin{Remark}
    As discussed in \cite[Remark 2.4]{SS} for uniformly elliptic equations, in the case of flat interface $\Gamma=\{x_d=0\}$, we retain more information about the contact sets $\{v=\mathcal{C}_v\}$ and thus can write an improved ABP in the form
    \[
    \sup_{B_1}u_-\leq \sup_{\partial B_1}u_{-}+C\big(\max_\Gamma g_++\|f_+^-\|_{L^d_\omega(\Omega^-\cap\{u=\mathcal{C}_u\})}+\|f_+^+\|_{L^d_\omega(\Omega^+\cap\{u=\mathcal{C}_u\})}\big).
    \]
\end{Remark}
\hfill
\section{Harnack inequality and H\"older regularity} \label{Section:Harnack}

\begin{Theorem}[Harnack Inequality]\label{Thm:Harnack}
    Let $u$ satisfy
    \begin{align*}
        \begin{cases}
            u\in S_\omega^*(f^\pm),\quad &\mbox{ in } \Omega^\pm,\\
            u_\nu^+-u_\nu^-=g,&\mbox{ on } \Gamma,
        \end{cases}
    \end{align*}
    with $f^\pm\in C(\Omega^\pm)\cap L^\infty(B_1)$, $g\in L^\infty(\Gamma)$ and $\Psi\in C^{1,\alpha_0}(\overline{B}_1')$. Assume further that $\|u\|_{L^\infty(B_1)}\leq 1$ , $u(\overline{x})\geq 0$, $\overline{x}=\frac{1}{5}e_n$, $B_{1/20}(\overline{x})\subset \Omega^+$. There exists $0<\varepsilon_0,c<1$ depending on $d, \lambda, \Lambda, [\Psi]_{C^{1,\alpha_0}}$ such that, if $\|g\|_{L^\infty(\Gamma)}+\|f^+\|_{L^d_\omega(\Omega^+)}+\|f^-\|_{L^d_\omega(\Omega^-)}\leq \varepsilon_0$, then
    \begin{align*}
        \inf_{B_{1/3}}u\geq -1+c.
    \end{align*}
\end{Theorem}

\begin{proof}
    Since $u+1\geq 0$, by the interior Harnack inequality obtained in \cite{JS} applied  in $B_{1/20}(\overline{x})$,
    \begin{align*}
        \sup_{B_{1/20}(\overline{x})}(u+1)\leq C\left(\inf_{B_{1/20}(\overline{x})}(u+1)+\|f^+\|_{L^d_\omega(\Omega^+)}\right),
    \end{align*}
where $C=C(d,\lambda,\Lambda,\omega)$. Then 
\begin{align*}
    1\leq u(\overline{x})+1\leq \sup_{B_{1/20}(\overline{x})}(u+1)\leq C\left(u(x)+1+\varepsilon_0\right),
\end{align*}
for all $x\in B_{1/20}(\overline{x})$ and thus
\begin{align}\label{eq:hlb}
    u\geq -1+\Tilde{c}, \quad\mbox{ in } B_{1/20}(\overline{x})
\end{align}
with $\Tilde{c}=1/C-\varepsilon_0<1$, $\varepsilon_0<1/C$. For $x\in D:=B_{3/4}(\overline{x})\setminus B_{1/20}(\overline{x})$, define
\begin{align*}
    &v(x)=\eta\phi(r)+\frac{\varepsilon_0}{c_0}w(x),\\
    &\phi(r)=r^{-\gamma}-\left(\frac{2}{3}\right)^{-\gamma}, \quad r=|x-\overline{x}|
\end{align*}
and $w$ is the unique viscosity solution to
\begin{align*}
    \begin{cases}
        \mathcal{M}^-(D^2w)=0, \quad & \mbox{ in } B_2\setminus \Tilde{\Gamma},\\
    w=\phi, &\mbox{ on } \Tilde{\Gamma},\\
    w=1, & \mbox{ on } \partial B_2
    \end{cases}
\end{align*}
where $\Tilde{\Gamma}=\{x_d=\Psi(x'), \, x'\in B_2'\}$. Furthermore, we have $w_\nu\geq c_0$ for $c_0>0$ depending only on $d, \lambda, \Lambda, \alpha_0$,  $[\Psi]_{C^{1,\alpha_0}}$, see \cite[Lemma 2.3]{SS}.

For any $x\in D^\pm=\Omega^\pm\cap D$, we proceed as in \cite{SS} to get
\begin{align*}
    \omega(x)\mathcal{M}^-(D^2v^\pm(x))\geq \omega(x)\eta\gamma r^{-\gamma-2}\left(\lambda(\gamma+1)-\Lambda(d-1)\right)\geq 0
\end{align*}
    for $x\in \Gamma\cap D$, it follows that
\begin{align*}
    v_\nu^+(x)-v_\nu^-(x)=\frac{\varepsilon_0}{c_0}\left(w_\nu^+(x)-w_\nu^+(x)\right)\geq 2\varepsilon_0>\|g\|_{L^\infty(\Gamma)}\geq g(x)
\end{align*}
since $\phi(r)$ is smooth outside $B_{1/20}(\overline{x})$.

Take $\eta,\varepsilon_0$ such that $v\leq \Tilde{c}$ on $\partial B_{1/20}(\overline{x})$ and $v\leq 0$ on $\partial B_{3/4}(\overline{x})$.

Note that $\phi(r)\geq 0$ in $0<r\leq 2/3$ and $\phi(r)\leq 0$ if $r\leq 2/3$. First fix $\eta$ such that $\eta\leq \Tilde{c}/(2\phi(1/20))$, then $\varepsilon_0$ such that $\varepsilon_0\leq c_0^{-1}\min\left\{\Tilde{c}/2,-\eta\phi(3/4)\right\}$. By \eqref{eq:hlb} we obtain $v\leq u+1$ on $\partial D$.

Since $u+1\in \overline{S}_\omega(|f^\pm|)$ in $D^\pm$, $v^\pm\in C^2(D^\pm)$ and $\omega(x)\mathcal{M}^-(D^2v^\pm)\geq 0$ in $D^\pm$, we get $u+1-v\in \overline{S}_\omega(|f^\pm|)$ in $D^\pm$. Also
\begin{align*}
    (u+1-v)_\nu^+-(u+1-v)_\nu^-\leq 0 \mbox{ on } \Gamma\cap D
\end{align*}
in the viscosity sense. Hence, applying Theorem \ref{Thm:deg_ABP} to $u+1-v$ in $D$ with $g=0$ we get
\begin{align*}
    \sup_D\,(u+1-v)_-\leq \sup_{\partial D}\,(u+1-v)_-+C\left(\|f^-\|_{L^d_\omega(\Omega^-)}+\|f^+\|_{L^d_\omega(\Omega^+)}\right) \leq C\varepsilon_0.
\end{align*}
Therefore $u\geq -1+v-C\varepsilon_0$ in $D$.

Moreover, for any $x\in B_{1/3}(0)\setminus B_{1/20}(\overline{x})$, we have $v(x)\geq \eta\phi(23/60)=c_1>0$, which depends only on $d, \lambda, \Lambda$. Choosing $\varepsilon_0$ such that $\varepsilon_0\leq c_1/(2C)$ we get $u\geq -1+c_1/2$ in $B_{1/3}(0)\setminus B_{1/20}(\overline{x})$. Therefore, by choosing $c=\min\{\Tilde{c}, c_1/2\}$ we get
\begin{align*}
    \inf_{B_{1/3}}\, u\geq -1+c.
\end{align*}
    
\end{proof}

Interior H\"older regularity follows exactly as in \cite{SS}.

\begin{Theorem}[Interior H\"older regularity]\label{Thm:int_Holder}
   Let $u$ satisfy
    \begin{align*}
        \begin{cases}
            u\in S_\omega^*(f^\pm),\quad &\mbox{ in } \Omega^\pm,\\
            u_\nu^+-u_\nu^-=g,&\mbox{ on } \Gamma,
        \end{cases}
    \end{align*}
    with $f^\pm\in C(\Omega^\pm)\cap L^\infty(B_1)$, $g\in L^\infty(\Gamma)$ and $\Psi\in C^{1, \alpha_0}(\overline{B}_1')$. Then $u\in C^{\beta_0}(\overline{B}_{1/2})$ and
    \begin{align*}
        \|u\|_{C^{\beta_0}(\overline{B}_{1/2})}\leq C\left(\|u\|_{L^\infty(B_1)}+\|g\|_{L^\infty(\Gamma)}+\|f^+\|_{L^d_\omega(\Omega^+)}+\|f^-\|_{L^d_\omega(\Omega^-)}\right)
    \end{align*}
    where $\beta_0$ and $C$ depend only on $d, \lambda, \Lambda, \alpha_0, [\Psi]_{C^{1,\alpha_0}}$  and $\omega$.
\end{Theorem}

We also obtain global H\"older continuity in the usual way, see for example \cite{CC}.
\begin{Proposition}[Global H\"older regularity]\label{Pro:Holder_global}
 Assume that $u\in C(B_1)$ is a viscosity solution to
    \begin{align*}
        \begin{cases}
            u\in S_\omega &\mbox{ in } \Omega^\pm\\
            u_\nu^+-u_\nu^-=g &\mbox{ on } \Gamma\\
            u=\varphi &\mbox{ on } \partial B_1
        \end{cases}
    \end{align*}
    with $g\in L^\infty(\Gamma)$ and $\supp(g)\subset \Gamma\cap B_{1-2r}$ for some $0<r<1/4$, $\varphi\in C^{\alpha_0}(\partial B_1)$ and $\Psi\in C^{1, \alpha_0}(\overline{B_1'})$. Then $u\in C^{\beta}(\overline{B_1})$, with $0<\beta\leq \min\{\beta_0,\alpha_0/2\}$ and
    \[
    \|u\|_{C^{\beta}(\overline{B_1})}\leq \frac{C}{r^\gamma}(\|\varphi\|_{C^{\alpha_0}(\partial B_1)}+\|g\|_{L^\infty(\Gamma)}),
    \]
    where $\gamma= \max \{ \alpha_0,\beta_0\}$ and $C$ depends only on $d,\lambda,\Lambda,\alpha_0, [\Psi]_{C^{1,\alpha_0}}$  and $\omega$.
\end{Proposition}

\section{Flat interface problems}\label{Section:flat}

In this section we do a comprehensive study of flat interface problems degenerating as $\omega(x)=|x_d|^{a}$, with $a$ constant, $0 <a< 1$.
\begin{align}\label{eq:flat}
    \begin{cases}
        \omega(x)F^\pm(D^2u)=f^\pm, &\mbox{ in } B_1^\pm\\
        u_{x_d}^+-u_{x_d}^-=g, &\mbox{ on } T=B_1\cap\{x_d=0\}
    \end{cases}
\end{align}

\subsection{Viscosity solutions}
In this section we use both notions of $C$ and $L^p$ viscosity solutions to better describe the behavior of solutions to \eqref{eq:flat}. 
\begin{Definition}[$C$ viscosity solution for flat problem]\label{Def:visc-flat} 
    We say that $u\in USC(B_1)$ is a viscosity subsolution to \eqref{eq:flat} in $B_1$ if for any $\varphi$ touching $u$ by above at $x_0$ in $B_1$, the following holds:
    \begin{enumerate}[label=(\roman*)]
        \item if $x_0\in B_1^\pm$ and $\varphi\in C^2(B_\delta(x_0))$, then $$ \omega(x_0) F^\pm(D^2\varphi(x_0))\geq f^\pm(x_0);$$
        \item if $x_0\in T$ and $\varphi\in C^1(\overline{B_\delta^+(x_0)})\cap C^1(\overline{B_\delta^-(x_0)})$, then
        \[
        \varphi_{x_d}^+(x_0)-\varphi_{x_d}^-(x_0)\geq g(x_0)
        \]
        where $\varphi^\pm=\varphi_{|B_\delta^\pm(x_0)}$.
    \end{enumerate}
\end{Definition}

The last condition can be replaced by (see \cite{DFS2015B})
{\it 
\begin{enumerate}[label=(\roman*')]
  \setcounter{enumi}{1}
  \item Let $x_0\in T$ and 
  \[
    \varphi(x)=P(x')+p^+x_d^+-p^-x_d^-
  \]
  where $P$ is a quadratic polynomial and $p^\pm\in \Rr$. If $\varphi$ touches $u$ by above at $x_0$ then
  \[
    p^+-p^-\geq g(x_0).
  \]
\end{enumerate}
}

\begin{Lemma}
    In Definition \ref{Def:visc-flat} for subsolutions of \eqref{eq:flat}, we can replace \textit{(ii)} with the following statement: if $x_0\in T$ and $\psi\in W^{2,p}(\overline{B_\delta^+(x_0)})\cap W^{2,p}(\overline{B_\delta^-
    (x_0)})$ is of the form
\begin{align}\label{eq:equiv_psi}
 \psi(x)=P(x')+p^+x_d^+-p^-x_d^-+C|x_d|^{2-a}
\end{align}
and touches $u$ by above at $x_0$ then either
    \[
    \esslimsup_{x\to x_0}\left(\omega(x)F^\pm(D^2\psi(x))-f^\pm(x)\right)\geq 0,
    \]
    or
    \[
    p^+-p^-\geq g(x_0).
    \]
    A similar result holds for supersolutions to \eqref{eq:flat}.
\end{Lemma}
\begin{proof}
Since $P(x')+p^+x_d^+-p^-x_d^-+C|x_d|^{2-a}\in  C^1(\overline {B_1^+})\cap C^1(\overline {B_1^-})$, it is clear that if $u$ is a $C$ viscosity subsolution to \eqref{eq:flat}, then the statement is true. To prove the converse, take $x_0\in T$ and let $ \varphi=P(x')+p^+x_d^+-p^-x_d^-$ touch $u$ by above at $x_0$. We argue by way of contradiction, assuming that
\begin{align}\label{eq:4.2cont}
    p^+-p^-<g(x_0).
\end{align}
Define $\psi(x)=\varphi(x)+\eta |x_d|-L|x_d|^{2-a}$ for $x\in B_\tau(x_0)$, where $\eta, \tau, L>0$ are to be determined. For $\eta$ small and $L$ large fixed, choose $\tau<r$ such that $\eta |x_d|-L|x_d|^{2-a}\geq 0$ in $B_\tau(x_0)$. In particular, $\psi$ is of the form \eqref{eq:equiv_psi} and 
\begin{align*}
    \begin{cases}
        \psi(x_0)=\varphi(x_0)=u(x_0),\\
        \psi(x)\geq \varphi(x)\geq u(x), \quad x\in B_\tau(x_0).
    \end{cases}
\end{align*}
Thus $\psi$ is a valid test function touching $u$ by above at $x_0$ and so, by the assumption, one of the following holds
\begin{equation}\label{eq:4.3psi}
    \begin{aligned}
        &\esslimsup_{x\to x_0}\left(\omega(x)F^\pm(D^2\psi(x))-f^\pm(x)\right)\geq 0,\\
        &\psi_{x_d}^+(x_0)-\psi_{x_d}^-(x_0)\geq g(x_0).
    \end{aligned}
\end{equation}
By \eqref{eq:4.2cont}, choosing $\eta$ small, 
\begin{align*}
    \psi_{x_d}^+(x_0)-\psi_{x_d}^-(x_0)=\varphi_{x_d}^+(x_0)-\varphi_{x_d}^-(x_0)+2\eta<g(x_0).
\end{align*}
Therefore the first inequality in \eqref{eq:4.3psi} must hold. Let $E_d=e_d\otimes e_d$. Then
\[
D^2\psi = D^2_{x'}P-L C|x_d|^{-a}E_d
\]
where $C=C(a)$.  Therefore
\begin{align*}
    &\esslimsup_{x\to x_0}\left( |x_d|^{a}\mathcal{M}^+(D^2\psi(x))-f^\pm(x)\right)\\
    \leq \, &\esslimsup_{x\to x_0}\left(  |x_d|^{a}\left(\Lambda (d-1)|D^2_{x'}P|-LC|x_d|^{-a}\right)-f^\pm(x)\right)\\
    \leq \,& -LC+\|f^\pm\|_{L^\infty}<0
\end{align*}
choosing $L>0$ sufficiently large. Here $C=C(d,\lambda,a)$. However, by ellipticity, and \eqref{eq:4.3psi},
\begin{align*}
    0\leq &\,\esslimsup_{x\to x_0}\left( |x_d|^{a}F^\pm(D^2\psi(x))-f^\pm(x)\right)\\
    \leq &\, \esslimsup_{x\to x_0}\left( |x_d|^{a}\mathcal{M}^+(D^2\psi(x))-f^\pm(x)\right)
\end{align*}
which is a contradiction.

\end{proof}

\subsection{Lower and upper $\varepsilon$-envelopes}

As is usual in the literature, we will use a family of regularizations in the $x'$-direction which was introduced in \cite{DFS2018}.
\begin{Definition}
    Given $u\in USC(B_1)$ and any $\varepsilon>0$, we define the upper $\varepsilon$-envelope of $u$ in the $x'$-direction as
    \[
    u^\varepsilon(y',y_d)=\sup_{x\in \overline{B_\rho}\cap \{x_d=y_d\}}\big\{u(x',y_d)-\frac{1}{\varepsilon}|x'-y'|^2\big\}
    \]
    for $y=(y',y_d)\in \overline{B_\rho}\subset B_1$. Similarly, given $u\in LSC(B_1)$, we define the lower $\varepsilon$-envelope of $u$ in the $x'$-direction as
    \[
    u_\varepsilon(y',y_d)=\inf_{x\in \overline{B_\rho}\cap \{x_d=y_d\}}\big\{u(x',y_d)+\frac{1}{\varepsilon}|x'-y'|^2\big\}
    \]
    for $y=(y',y_d)\in \overline{B_\rho}\subset B_1$.
\end{Definition}
Note that there is $y_\varepsilon\in \overline{B_\rho}\cap \{x_d=y_d\}$ such that
\[
u^\varepsilon(y)=u(y_\varepsilon)-\frac{1}{\varepsilon}|y_\varepsilon'-y'|^2
\]
with $|y'-y_\varepsilon'|\leq \sqrt{2\varepsilon\|u\|_\infty}$, since $u^\varepsilon(y)\geq u(y)$ and
\[
\frac{1}{\varepsilon}|y_\varepsilon'-y'|^2=u(y_\varepsilon)-u^\varepsilon(y)\leq u(y_\varepsilon)-u(y).
\]

The following result was proven in \cite[Lemma 3.1]{DFS2018}.
\begin{Lemma}\label{Lm:prop_env}
    The following properties hold:
    \begin{enumerate}
        \item $u^\varepsilon\geq u$ in $B_\rho$ and $\limsup_{\varepsilon\to 0}u^\varepsilon=u$;
        \item $u^\varepsilon\in C^{0,1}_{y'}(\overline{B_\rho})$, with $[u^\varepsilon]_{C^{0,1}_{y'}(\overline{B_\rho})}\leq 6\rho/\varepsilon$;
        \item $u^\varepsilon\in C^{1,1}_{y'}$ by below in $B_\rho$. Thus, $u^\varepsilon$ is punctually second order differentiable in the $x'$-direction almost everywhere in $B_\rho$.
    \end{enumerate}
\end{Lemma}

\begin{Proposition}\label{Pro:S-env}
    Let $f^\pm\in C(B_1^\pm)$ and $g\in C(T)$. If $u$ is a bounded viscosity subsolution to \eqref{eq:flat} then, for any $\varepsilon>0$ small, $u^\varepsilon$ is a viscosity subsolution to
    \begin{align*}
        \begin{cases}
            \omega(x) F^\pm(D^2u^\varepsilon)=f_\varepsilon^\pm, & \mbox{ in } B_r^\pm,\\
            (u^\varepsilon)_{x_d}^+-(u^\varepsilon)_{x_d}^-=g_\varepsilon, & \mbox{ on } T_r=B_r\cap \{x_d=0\},
        \end{cases}
    \end{align*}
    with $r\leq \rho-r_\varepsilon$, $r_\varepsilon=(2\varepsilon\|u\|_{L^\infty(B_1)})^{1/2}$, $f_\varepsilon^\pm=f-\gamma_{f^\pm}(r_\varepsilon)$, and $g_\varepsilon=g-\gamma_g(r_\varepsilon)$.
\end{Proposition}
\begin{proof}
    The proof consists of two steps. First we prove that $\omega(x)F^\pm(D^2u^\varepsilon)=f_\varepsilon^\pm$ in $B_r^\pm$, following a similar reasoning to \cite[Lemma 3.1]{DFS2018}. Then, the proof of the transmission condition $(u^\varepsilon)_{x_d}^+-(u^\varepsilon)_{x_d}^-=g_\varepsilon$ on $T_r$, follows from exactly the same argument as in \cite{SS}.

    For the first part, let $\varphi\in C^2(B_r)$ touch $u^\varepsilon$ by above at $\bar x\in B_r^+$. Then, for $\varepsilon$ small,
    \[
    u^\varepsilon(\bar x) = u(\bar x_\varepsilon)-\frac{1}{\varepsilon}|\bar x_\varepsilon'-\bar x'|^2,
    \]
    with $|\bar x_\varepsilon'-\bar x'|^2\leq 2\varepsilon\|u\|_\infty$. Consider the function
    \[
    \Phi(y)=\varphi(y+\bar x-\bar x_\varepsilon)+\frac{1}{\varepsilon}|\bar x_\varepsilon'-\bar x|^2.
    \]
    with our choice of $r$ and $y\in B_\rho^+$ close enough to $\bar x_\varepsilon$, the point $y+\bar x-\bar x_\varepsilon\in B_\rho^+$. Thus, by the definition of $u^\varepsilon$,
    \[
    u(y)\leq u^\varepsilon(y+\bar x-\bar x_\varepsilon)+\frac{1}{\varepsilon}|\bar x_\varepsilon-\bar x'|^2
    \]
and so,
\[
u(y)\leq \varphi(y+\bar x-\bar x_\varepsilon)+\frac{1}{\varepsilon}|\bar x_\varepsilon'-\bar x'|^2,
\]
with equality at $y=\bar x_\varepsilon$, since $\varphi(\bar x)=u ^\varepsilon (\bar x)$. Thus, the function $\Phi$ touches $u$ by above at $\bar x_\varepsilon$. Therefore,
\begin{align*}
    \omega(\bar x) F^+(D^2\varphi(\bar x))&=\omega(\bar x)F^+(D^2\Phi(\bar x_\varepsilon))\\
&\geq\, f^+(x_\varepsilon)\geq f^+(\bar x)-\gamma_{f^+}\left(\sqrt{2\varepsilon\|u\|_\infty}\right)=f^*_\varepsilon(\bar x)
\end{align*}
as intended. 

The transmission condition follows identically to \cite{SS} since it is independent of the PDE.

\end{proof}

\subsection{Comparison principle and uniqueness}

We will also make use of the classical notion of the following half-relaxed limits. Let $\{u_k\}_k$ be a sequence of functions. For $x\in \overline{B_1}$, we define
\[
    {\limsup}^*\, u_k(x)=\lim_{j\to \infty}\sup\left\{ u_k(y)\,:\,k\geq j, \, y\in \overline{B_1},\, \mbox{ and } |y-x|\leq \frac{1}{j}\right\}.
\]
Similarly, for $x\in \overline{B_1}$, we define
\[
    {\liminf}_*\, u_k(x)=\lim_{j\to \infty}\inf\left\{ u_k(y)\,:\,k\geq j, \, y\in \overline{B_1},\, \mbox{ and } |y-x|\leq \frac{1}{j}\right\}.
\]
Then ${\limsup}^*\, u_k\in USC(\overline{B_1})$ and ${\liminf}_*\, u_k\in LSC(\overline{B_1})$. We have the following lemma from \cite{CIL}.
\begin{Lemma}\label{Lm:limits}
    Let $\{u_k\}_k\subset USC(\overline{B_1})$ and $u={\liminf}_*\, u_k$. Fix $x_0\in \overline{B_1}$. If a continuous function $\varphi$ touches $u$ from above at $x_0$, then there exist indexes $k_j\to \infty$, points $x_j\in \overline{B_1}$, and functions $\varphi_j\in C$ such that $\varphi_j$ touches $u_{k_j}$ by above at $x_j$,
    \[
    x_j\to x_0, \quad \mbox{ and } \quad u_{k_j}(x_j)\to u(x_0), \quad \mbox{ as } j\to \infty.
    \]
    Moreover
    \[
    \varphi_j(x)=\varphi(x)-\varphi(x_j)+u_{k_j}(x_j)+\delta(|x-x_0|^2-|x_j-x_0|^2),
    \]
    for an arbitrary $\delta>0$.
\end{Lemma}

The following result can be found in \cite[Corollary 1.8]{MW}, which we present in a simplified form. 
\begin{Corollary}\label{Cor:MaWang}
    Let $u$ be a viscosity solution to
    \[
    \begin{cases}
        F(D^2u)=f &\mbox{ in } Q_1^+\\
        u= \phi(x) &\mbox{ on } Q_1',
    \end{cases}
    \]
    where $F$ is uniformly elliptic, $\phi$ is $C^{1,\alpha}$ at $0$, and 
    \begin{align}\label{eq:MaWang_f}
        r\left(\intav{Q_r^+}f^p\right)^\frac{1}{p}\leq C r^\alpha.
    \end{align}
    Then $u$ is $C^{1,\alpha}$ at $0$, that is, there exists an affine function $\ell$ such that for every $0<r<1/2$,
    \begin{align*}
        \sup_{Q_r^+}|u-\ell|\leq C r^{1+\alpha}.
    \end{align*}
\end{Corollary}

\begin{Theorem}\label{Thm:S-dif}
    Let $f_1^\pm, f_2^\pm\in C(B_1^\pm)\cap L^\infty(B_1^\pm)$  and $g_1,g_2\in C(T)$. Assume that $u\in USC(\overline{B_1})$ and $v\in LSC(\overline{B_1})$ are bounded and satisfy 
    \begin{align*}
        \begin{cases}
            \omega(x)F^\pm(D^2u)\geq f_1^\pm, & \mbox{ in } B_1^\pm\\
            u_{x_d}^+-u_{x_d}^-\geq g_1, &\mbox{ on } T
        \end{cases}
    \end{align*}
    and
    \begin{align*}
        \begin{cases}
            \omega(x)F^\pm(D^2v)\leq f_2^\pm, & \mbox{ in } B_1^\pm\\
            v_{x_d}^+-v_{x_d}^-\leq g_2, &\mbox{ on } T
        \end{cases}
    \end{align*}
    in the viscosity sense. Then, $w=u-v$ satisfies
    \begin{align*}
        \begin{cases}
            \omega(x)\mathcal{M}^+(D^2w)\geq f_1^\pm-f_2^\pm, & \mbox{ in } B_1^\pm\\
            w_{x_d}^+-w_{x_d}^-\geq g_1-g_2, &\mbox{ on } T
        \end{cases}
    \end{align*}    
    in the viscosity sense.
    
\end{Theorem}

\begin{proof}
    The fact that $\omega(x)\mathcal{M}^+(D^2w)\geq f_1^\pm-f_2^\pm$ in $B_1^\pm$ follows from the argument in \cite[Theorem 4]{JS}, noting that within either $B_1^+$ or $B_1^-$, the notions of $C$ and $L^p$ viscosity solutions coincide.

    We need to show the transmission condition. Let $x_0=(x_0',0)\in T$ and assume that $P(x')+p^+x_d^+-p^-x_d^-$ touches $w$ by above at $x_0$, with $P$ being a quadratic polynomial and $p^\pm\in \Rr$. We aim at showing that 
    \begin{align}\label{eq:4.4aim}
        p^+-p^-\geq g_1(x_0)-g_2(x_0).
    \end{align}
    Fix $\tau>0$ and $C>0$ large to be chosen. Then the $W^{2,p}$ test function 
    \[
    \varphi(x)=P(x')+(p^++\tau)x_d^+-(p^--\tau)x_d^--C|x_d|^{2-a}
    \]
    touches $w$ strictly by above at $x_0$, possibly in a smaller neighborhood where  $$ x_d^\pm-C(x_d^\pm)^{2-a}\geq 0.$$ 

    For $\varepsilon>0$, consider the upper and lower-envelopes $u^\varepsilon$ and $v_\varepsilon$ and take $w_\varepsilon=u^\varepsilon-v_\varepsilon$. By Lemma \ref{Lm:prop_env} \textit{(1)},
    \[
    \limsup_{\varepsilon\to 0}w_\varepsilon=w.
    \]
    By Lemma \ref{Lm:limits}, there are points $x_\varepsilon\in B_1$ with $x_\varepsilon\to x_0$, up to a subsequence, and functions
    \[
    \varphi_\varepsilon(x)=\varphi(x)-\varphi(x_\varepsilon)+w_\varepsilon(x_\varepsilon)+|x-x_0|^2-|x_\varepsilon-x_0|^2
    \]
    such that $\varphi_\varepsilon$ touches $w_\varepsilon$ strictly from above at $x_\varepsilon$. Particularly, for $\delta>0$ sufficiently small, there is $\eta>0$ such that $\varphi_\varepsilon-w_\varepsilon\geq \eta>0$ on $\partial B_\delta(x_\varepsilon)$. By Proposition \ref{Pro:S-env}, $w_\varepsilon$ satisfies, in the viscosity sense,
    \begin{align}\label{eq:4.5w_eps}
        \omega(x)\mathcal{M}^+(D^2w_\varepsilon)\geq (f_1^\pm)_\varepsilon-(f_2^\pm)_\varepsilon, \quad \mbox{ in } B_\rho^\pm,
    \end{align}
    for some $0<\rho<1$ such that $\overline{B_\delta(x_\varepsilon)}\subset B_\rho$.

    Now we note that $D^2\varphi=D^2_{x'}P-C(2-a)(1-a) |x_d|^{-a}e_d\otimes e_d$ and recalling that $\omega(x)=|x_d|^{a}$ we get
\begin{equation}\label{eq:4.6phi}
    \begin{aligned}
        |x_d|^{a}\mathcal{M}^+(D^2\varphi_\varepsilon)\leq &\,|x_d|^{a}\left(\|D^2_{x'}P\|+2\Lambda\right)-C(2-a)(1-a)\\
        < &\inf_{B_\rho^\pm}\left[(f_1^\pm)_\varepsilon-(f_2^\pm)_\varepsilon\right], \quad \mbox{ in } B_\rho^\pm,
    \end{aligned}
\end{equation}
    provided we take $C$ large enough. Note that this immediately implies that $x_\varepsilon\in T$ since $\varphi_\varepsilon$ touches $w_\varepsilon$ by above at $x_\varepsilon$ and if $x_\varepsilon$ were in $B_\rho^\pm$ then this inequality would contradict \eqref{eq:4.5w_eps}, by the definition of $L^p$ viscosity solution.

    Define
    \begin{align}\label{eq:4.6psi}
        \psi=\varphi_\varepsilon-w_\varepsilon-\eta/2,
    \end{align}
    with $\psi\geq \eta/2>0$ on $\partial B_\delta(x_\varepsilon)$ and $\psi(x_\varepsilon)=-\eta/2<0$. Let $\mathcal{C}_\psi$ be the convex envelope of $-\psi_-$ in $B_{2\delta}'(x_\varepsilon)$, where we have extended $-\psi_-\equiv 0$ outside of $\overline{B_{\delta}'(x_\varepsilon)}$. By Lemma \ref{Lm:prop_env} \textit{(3)}, we know that $\psi\in C^{1,1}_{x'}$ by above in $B_\rho$ (since the bad term $-|x_d|^{2-a}$ in the definition of $\psi$ is non-positive). Hence, for any $x_0'\in \overline{B_\delta'(x_\varepsilon)}$, there exists a paraboloid $P(x')$ with uniform opening that touches $\psi(x',0)$ by above at $x_0'$. We have $\mathcal{C}_\psi\in C_{x'}^{1,1}(\overline{B_\delta'(x_\varepsilon)})$ and for any $t>0$, we claim that
    \[
    |D_t|:=|\{x'\in \overline{B_\delta'(x_\varepsilon)} \, : \, \mathcal{C}_\psi(x')=\psi(x',0) \mbox{ and } |D_{x'}\mathcal{C}_\psi(x')|\leq t\}|>0.
    \]
    Indeed, the fact that the set of contact points $\{\psi=\mathcal{C}_\psi\}$ in $\overline{B'_\delta(x_\varepsilon)}$ has positive measure follows from the Alexandroff Lemma (see \cite[Lemma 3.5]{CC}) which implies that there exists $A\subset B'_\delta(x_\varepsilon)$ such that $|B'_\delta(x_\varepsilon)\setminus A|=0$ and
    \[
    0<\eta/2= \sup_{B'_\delta(x_\varepsilon)}\psi_-\leq \left(\int_{A\cap \{\psi=\mathcal{C}_\psi\}}\det D^2\mathcal{C}_\psi\,dx\right)^\frac{1}{d}.
    \]
    Since $x_\varepsilon$ is a minimum of $\psi$ we have  
 $\mathcal{C}_\psi(x'_\varepsilon)=\psi(x'_\varepsilon,0)$ and $D_{x'}\mathcal{C}_\psi(x'_\varepsilon)=0$, thus $x_\varepsilon\in D_t\neq \emptyset$, for any $t>0$. Since the gradient of $\mathcal{C}_\psi$ is continuous, the claim follows. Hence, choosing $t\leq \eta/(4\delta)$, there exists $y_\varepsilon'\in D_t$ such that both $u^\varepsilon$ and $v_\varepsilon$ are punctually second order differentiable at $y_\varepsilon=(y_\varepsilon',0)$ in the $x'$-direction and such that
 \[
 \ell(x')=D_{x'}\mathcal{C}_\psi(y'_\varepsilon)\cdot (x'-y_\varepsilon')+\psi(y_\varepsilon)
 \]
 touches $\psi$ from below at $y_\varepsilon$ on $\overline{B_\delta(x_\varepsilon)}$. Furthermore by \eqref{eq:4.6psi} we get $w_\varepsilon\leq \varphi_\varepsilon-\ell-\eta/2$ on $\partial B_\delta^\pm(x_\varepsilon)$ and by \eqref{eq:4.5w_eps} and \eqref{eq:4.6phi} we get
 \[
 \omega(x)\mathcal{M}^+(D^2w_\varepsilon)>\omega(x)\mathcal{M}^+(D^2(\varphi_\varepsilon-\ell-\eta/2)),\quad \mbox{ in } B_\delta^\pm(x_\varepsilon).
 \]
 Therefore by the comparison principle which follows from \cite[Theorem 4]{JS} applied to each part of $B^\pm_\delta(x_\varepsilon)$, we get that $w_\varepsilon\leq \varphi_\varepsilon-\ell-\eta/2$ on $\overline{B_\delta(x_\varepsilon)}$. Let now
 \[
 \bar \varphi = \varphi_\varepsilon-\ell-\eta/2.
 \]
 Consider the viscosity solutions $\bar u^\varepsilon$ and $\bar v_\varepsilon$ to the Dirichlet problems
 \begin{align*}
     \begin{cases}
         \omega(x) F^\pm(D^2\bar u^\varepsilon)=\left(f^\pm_1\right)_\varepsilon &\mbox{ in } B_\delta^\pm(x_\varepsilon),\\
         \bar u^\varepsilon=u^\varepsilon & \mbox{ on } \partial B_\delta^\pm(x_\varepsilon).
     \end{cases}
 \end{align*}
 and
  \begin{align*}
     \begin{cases}
         \omega(x) F^\pm(D^2\bar v_\varepsilon)=\left(f^\pm_2\right)_\varepsilon &\mbox{ in } B_\delta^\pm(x_\varepsilon),\\
         \bar v_\varepsilon=v_\varepsilon & \mbox{ on } \partial B_\delta^\pm(x_\varepsilon).
     \end{cases}
 \end{align*}
 By applying again the comparison principle, $\bar u^\varepsilon\geq u^\varepsilon$ and $\bar v_\varepsilon\leq v_\varepsilon$ in $B_\delta(x_\varepsilon)$, and thus
\begin{align}\label{eq:4.8uv}
     \left(\bar u^\varepsilon\right)^+_{x_n}-\left(\bar u^\varepsilon\right)^-_{x_n}\geq \left( g_1\right)_\varepsilon \quad \mbox{ and }\quad \left(\bar v_\varepsilon\right)^+_{x_n}-\left(\bar v_\varepsilon\right)^-_{x_n}\leq \left( g_2\right)_\varepsilon
\end{align}
 on $B_\delta(x_\varepsilon)\cap \{x_d=0\}$, in the viscosity sense, where 
 \[
  \left( g_1\right)_\varepsilon=g_1-\gamma_{g_1}\left((2\varepsilon\|u\|_{L^\infty(B_1)})^{1/2}\right)
 \]
 and
 \[
 \left( g_2\right)_\varepsilon=g_2-\gamma_{g_2}\left((2\varepsilon\|v\|_{L^\infty(B_1)})^{1/2}\right).
 \]
 By Lemma \ref{Lm:prop_env} \textit{(3)}, in particular we have $u^\varepsilon, v_\varepsilon\in C^{1,\alpha}_{x'}(y_\varepsilon)$. We are now in conditions to apply Corollary \ref{Cor:MaWang}, since $f:=\left(f^\pm_i\right)_\varepsilon\omega^{-1}$ satisfies condition \eqref{eq:MaWang_f} for $\alpha=1-a$, and thus we get pointwise $C^{1,\alpha}$ estimates (depending on $\varepsilon$), which imply the existence of $r_0>0$ and linear polynomials $\ell_u^\pm$ and $\ell_v^\pm$ such that
 \begin{equation*}
     \begin{aligned}
         \|(\bar u^ \varepsilon)^\pm-\ell_u^\pm\|_{L^\infty(B^\pm_r(y_\varepsilon))}\leq Cr^{1+\alpha}\\
         \|\bar v_ \varepsilon^\pm-\ell_v^\pm\|_{L^\infty(B^\pm_r(y_\varepsilon))}\leq Cr^{1+\alpha}
     \end{aligned}\quad  \mbox{ for all } 0<r<r_0. 
 \end{equation*}
To simplify notation, call $p_u^\pm=D\ell_u^\pm\cdot e_d$ and $p_v^\pm=D\ell_v^\pm\cdot e_d$. We can now argue similarly as in \cite[Lemma 4.3]{DFS2018}, noting that the source term is allowed to be unbounded. We thus get that \eqref{eq:4.8uv} holds pointwise, that is,
\begin{align}\label{eq:4.9puv}
    p_u^+-p_u^-\geq (g_1)_ \varepsilon(y_\varepsilon) \quad \mbox{ and } \quad p_v^+-p_v^-\leq (g_2)_\varepsilon(y_\varepsilon).
\end{align}
Let $\bar w_\varepsilon=\bar u_\varepsilon-\bar v_\varepsilon$. Then, by the previous computations, we have
\begin{align*}
    \begin{cases}
        \omega(x)\mathcal{M}^+(D^2\bar w_\varepsilon)\geq \omega(x)\mathcal{M}^+(D^2\bar \varphi) & \mbox{ in } B_\delta^\pm(x_\varepsilon)\\
        \bar w_\varepsilon\leq \bar \varphi & \mbox{ on } \partial B_\delta^\pm(x_\varepsilon).
    \end{cases}
\end{align*}
It follows that $\bar w_\varepsilon\leq \bar \varphi$ in $B_\delta(x_\varepsilon)$ and $\bar w_\varepsilon(y_\varepsilon)=\bar\varphi(y_\varepsilon)$. Since $\bar w_\varepsilon\in C^{1,\alpha}(y_\varepsilon)$, we have that
\begin{align*}
    &p^++\tau =\bar \varphi^+_{x_n}(y_\varepsilon)\geq \left(\bar w_\varepsilon\right)_{x_d}^+(y_\varepsilon)=p_u^+-p_v^+,\\
    &p^--\tau =\bar \varphi^-_{x_n}(y_\varepsilon)\leq \left(\bar w_\varepsilon\right)_{x_d}^-(y_\varepsilon)=p_u^--p_v^-.
\end{align*}
Therefore, combining the previous estimates with \eqref{eq:4.9puv}, yields
\begin{align*}
    p^+-p^-+2\tau \geq&\, (g_1)_\varepsilon(y_\varepsilon)-(g_2)_\varepsilon(y_\varepsilon)\\
    =&\,(g_1-g_2)(y_\varepsilon)+\gamma_{g_1}\left((2\varepsilon\|u\|_\infty)^{1/2}\right)-\gamma_{g_2}\left((2\varepsilon\|v\|_\infty)^{1/2}\right).
\end{align*}
We emphasize that although $p^\pm_u$ and $p^\pm_v$ depend on $\varepsilon$ and might explode as $\varepsilon\to 0$, they were only used in an intermediate step, and the final estimate we got above is stable under this limit. Thus, recalling that $y_\varepsilon\in B_\delta(x_\varepsilon)$ and $x_\varepsilon\to x_0$ as $\varepsilon\to 0$, we can start by letting $\tau \to 0$, then $\delta \to 0$, to get $y_\varepsilon\to x_\varepsilon$, and finally $\varepsilon\to 0$, and obtain the desired \eqref{eq:4.4aim}.
    
\end{proof}
These two results follow from  Theorem \ref{Thm:S-dif} and the ABP estimate(Theorem \ref{Thm:deg_ABP}).

\begin{Corollary}[Uniqueness]\label{Cor:unique}
    There exists a unique viscosity solution to 
    \begin{align*}
        \begin{cases}
        \omega(x)F^\pm(D^2u)=f^\pm &\mbox{ in } B_1^\pm\\
        u_{x_d}^+-u_{x_d}^-=g &\mbox{ on } T=B_1\cap\{x_d=0\}\\
        u=h &\mbox{ on } \partial B_1.
    \end{cases}
    \end{align*}
\end{Corollary}

\begin{Theorem}[Comparison]\label{Thm:comparison}
    Let $u,v \in C(\overline{B_1})$ be bounded viscosity sub and supersolutions of \eqref{eq:flat}, respectively. If $u \leq v$ on $\p B_1$ then $u\leq v$ in $B_1$. 
\end{Theorem}

The existence of solutions for \eqref{eq:flat}  follows the from same argument as in \cite[Theorem 4.11]{SS}. 

\begin{Theorem}[Existence]
    Let $f^\pm \in C(B_1^\pm \cup T) \cap L^\infty(B_1)$, $ g\in C(T)$, and $\phi\in C(\p B_1)$. Then there exists a unique viscosity solution $u \in C(\overline{B_1})$ of \eqref{eq:flat} such that $u= \phi$ in $\p B_1$.
 \end{Theorem}

\section{A stability result for $C^2$ interfaces} \label{Section:approximation}
In this section we prove that the equation is stable under small perturbations. The proof differs substantially from the uniformly elliptic case. This difficulty stems from the fact that as the interface changes, so does the degeneracy set, and thus using a correct $W^{2,p}$ test function becomes essential to overcome the degeneracy of the equation.

\begin{Lemma} \label{Lm:stability}
    Assume that $\Gamma_k\in C^2$ and $u_k\in C(B_1)$ satisfy
    \begin{align*}
        \begin{cases}
           |x_d - \Psi_k(x') |^{a_k(x)} F_k^\pm(D^2u_k)=f_k^\pm & \mbox{ in } \Omega_k^\pm\\
       (u_k^+)_\nu-(u_k^-)_\nu=g_k & \mbox{ on } \Gamma_k
        \end{cases}
    \end{align*}
    where $F_k^\pm$ satisfy assumption \ref{Assumption1}-\ref{Assumption3} with $\Gamma_k=B_1\cap \{x_d=\Psi_k(x')\}$ for $\Psi_k\in C^2$,  $\omega_k(x)=|x_d-\Psi_k(x')|^{a_k(x)}$, $f_k^\pm\in C(\Omega_k^\pm\cup \Gamma_k)$ and $g_k\in C(\Gamma_k)$, for $k\geq 1$. Suppose that there are continuous functions $u, f^\pm$ and $g$, a constant $a$, and elliptic operators $F^\pm\in \mathcal{E}(\lambda,\Lambda)$ such that, as $k\to \infty$, we have
    \begin{enumerate}
        \item $ F_k^\pm\to F^\pm$ uniformly on compact subsets of $S(d)$ 
        \item $u_k\to u$ uniformly on compact subsets of $B_1$;
        \item $\|f_k^\pm-f^\pm\|_{L^\infty(\Omega_k)}\to 0$;
        \item $\|g_k-g\|_{L^\infty(\Gamma_k)}=\sup_{x'\in B_1'}|g_k(x',\Psi_k(x'))-g(x',0)|\to 0$;
        \item $\Gamma_k\to T$ in $C^2$ in the sense that $\|\Psi_k\|_{C^2(B_1')}\to 0$;
        \item $a_k\to a$ uniformly on $B_1$.  
    \end{enumerate}
    Then $u\in C(B_1)$ is a viscosity solution to
    \begin{align*}
        \begin{cases}
            |x_d|^{a}F^\pm(D^2u)=f^\pm & \mbox{ in } B_1^\pm\\
            u_{x_d}^+-u_{x_d}^-=g &\mbox{ on } T.
        \end{cases}
    \end{align*}
\end{Lemma}

\begin{proof}    We only prove that $u$ is a viscosity subsolution since the other case follows similarly.  First, we show that
\[ |x_d|^{a}F^\pm(D^2u) \geq f^\pm  \quad \mbox{ in } B_1^{\pm}. \]
If, by contradiction, we suppose that it fails then there exists $x_0 \in B_1^{\pm}$ and a test function  $\varphi \in C^2(B_{\delta}(x_0))$ such that $\varphi$ touches $u$ from above at $x_0$, and
\[  |(x_0)_d|^{a}F^\pm(D^2\varphi(x_0)) < f(x_0)^\pm. 
\]
Without losing generality, we can assume that $x_0 \in B_1^+$ and that $\varphi$ touches $u$ strictly from above at $x_0$, up to considering $\varphi + \varepsilon |x-x_0|^{2}$ instead of $\varphi$, with $\varepsilon$ small. Now, since $u_k \to u$ uniformly on compact sets, there exists $\varepsilon_k >0$ such that $\varphi + \varepsilon_k \geq u_k$ in $\overline{ B_r(x_0)}$  for $k$ large and $r \leq \delta$ small.
We can consider $r$ small such that $\overline{ B_{r}(x_0)} \subset \Omega_k^{+}$, for some $k$ large enough, having $\Gamma_k \to T$.

Then, we define
\[ 
d_k= \inf_{B_{r_k}(x_0)} (\varphi + \varepsilon_k - u_k) \geq 0
\]
with $0 < r_k < r$  and $r_k \searrow 0$. Thus  $\overline{ B_{r_k}(x_0)} \subset \Omega_k^{+}$. 

Now, let $x_k \in \Omega_k^{+}$ be a point for which the infimum is attained, i.e., 
\[ d_k= \varphi(x_k)+ \varepsilon_k - u_k(x_k)  \]
and define $c_k= \varepsilon_k-d_k$, Then $x_k \to x_0, c_k \to 0$, and $\varphi + c_k $ touches $u_k$ from above at $x_k \in \Omega_k^+$, for $k$ large. Hence, since $|(x_k)_d - \Psi_k(x'_k)|^{a_k (x_k)}F_k^+(D^2 u_k(x_k)) \geq f_k^+$ in $\Omega_k^+$, we need to have
\[ |(x_k)_d - \Psi_k(x'_k)|^{a_k (x_k)} F_k^+(D^2 \varphi(x_k)) \geq f_k^+ (x_k).\]
Letting $k \to \infty$, we obtain
\[ |(x_0)_d|^{a}F^\pm(D^2\varphi(x_0)) \geq  f(x_0)^\pm. \]

    Now we prove the transmission condition. By contradiction, assume that there exists $x_0\in T$, $r>0$ small and $\varphi(x)=P(x')+p^+x_d^+-p^-x_d^-$ such that $\varphi$ touches $u$ from above at $x_0$ but 
    \begin{align}\label{eq:stab_contr}
        p^+-p^-<g(x_0).
    \end{align}
    Let $\psi(x)=\varphi(x)+\eta|x_d|-L|x_d|^{2-a}$, for $\eta, L>0$ to be fixed and take $\tau$ so small that $\eta |x_d|-L|x_d|^{2-a}\geq 0$ in $B_\tau(x_0)$. Then $\psi$ touches $u$ strictly from above at $x_0$ in $B_\tau(x_0)$. Arguing as before, there exist $c_k, r_k, x_k$ such that if we define
    \[
    \phi(x)=\psi(x',x_d-\Psi_k(x'))+c_k
    \]
then $\phi$ touches $u_k$ strictly from above at $x_k$ in $B_{r_k}(x_0)$, with $c_k\to 0$, $r_k\to 0$ and $x_k\to x_0$. Recall
\begin{align*}
    \phi(x)=P(x')+p^+(x_d-\Psi_k(x'))^+-p^-(x_d-\Psi_k(x'))^-\\
+\eta|x_d-\Psi_k(x')|-L|x_d-\Psi_k(x')|^{2-a}
\end{align*}
and note that
\begin{align*}
    D^2|x_d-\Psi_k(x')|^{2-a}=\begin{pmatrix}
        M' & p\\
        p^T & q
    \end{pmatrix}=:M
\end{align*}
with
\begin{align*}
    M'=&\,(2-a)(1-a)|x_d-\Psi_k(x')|^{-a}(D_{x'}\Psi_k(x')\otimes D_{x'}\Psi_k(x')) \\
    \mp&\,(1-a) |x_d-\Psi_k(x')|^{1-a} D^2_{x'}\Psi_k(x'),\\
    p=&\,-(2-a)(1-a)|x_d-\Psi_k(x')|^{-a}D_{x'}\Psi_k(x'),\\
    q=&\, (2-a)(1-a)|x_d-\Psi_k(x')|^{-a},
\end{align*}
here $\pm$ refers to whether $\pm(x_d-\Psi_k(x'))>0$. Since $\Psi_k\to 0$ in the $C^2$ norm, the leading term is $q$. Hence, choosing $k_0$ large enough, for $k\geq k_0$ we can assume that $M$ has a positive eigenvalue
\[
e\geq \frac{1}{2}(2-a)(1-a)|x_d-\Psi_k(x')|^{-a}.
\]

The above computation can be performed in the same manner if we consider, instead of $\phi$, a new family of test functions $\phi_k$ defined as
\begin{align*}
    \phi_k(x)=P(x')+p^+(x_d-\Psi_k(x'))^+-p^-(x_d-\Psi_k(x'))^-\\
+\eta|x_d-\Psi_k(x')|-L|x_d-\Psi_k(x')|^{2-(a+\varepsilon_k)}
\end{align*}
with $\varepsilon_k>0$ such that $a_k(x) \leq a + \varepsilon_k $ in $B_{r_k}(x_k)$, since $a_k(x) \to a$ uniformly as $\varepsilon_k \searrow 0$.

Now, we are presented with two options: either there exists $\bar k>0$ such that for every $k\geq \bar k$ we have $x_k\in \Omega_k^\pm$, in which case by the equivalence of  notions of $C$ and $L^p$ viscosity solutions, since $\phi_k$ touches $u_k$ by above at $x_k$, it must hold
\begin{align*}
    \esslimsup_{x\to x_k}\left(|x_d -\Psi_k(x')|^{a_k(x)}  F_k^\pm(D^2\phi_k)-f_k^\pm  \right) \geq 0;
\end{align*}
However, exploiting the ellipticity of $F_k$ and the explicit expression for $D^2\phi_k$, we have
\begin{align*}
&\esslimsup_{x\to x_k}\left(  |x_d -\Psi_k(x')|^{a_k(x)} F_k^\pm(D^2\phi_k)-f_k^\pm \right) \\
\leq \,&\esslimsup_{x\to x_k}\left( |x_d -\Psi_k(x')|^{a_k(x)} \mathcal{M}^+(D^2\phi_k) -f_k^\pm  \right)   \\ 
 \leq\,& \esslimsup_{x\to x_k}\left(|x_d-\Psi_k(x')|^{a_k(x)}C_1|D_{x'}^2P|-C_2 |x_d-\Psi_k(x')|^{a_k(x) - (a + \varepsilon_k) }L\right)+ \|f_k^\pm\|_{L^\infty}\\
 <\,&0
\end{align*}
provided we choose $L$ large enough, independent of $k$. This produces a contradiction.

The other possibility is that for every $\bar k\geq 1$ there exists $k\geq \bar k$ such that $x_k\in \Gamma_k$     and in this case we have
\begin{align*}
    \phi_{\nu_k}^+(x_k)-\phi_{\nu_k}^-(x_k)\geq g_k(x_0).
\end{align*}
By letting $\bar k\to \infty$ we obtain a contradiction with \eqref{eq:stab_contr}.  
\end{proof} 
\begin{Remark}
    Note that, up to this point, the results in this paper could be extended for operators with more general variable coefficients of the form $F^\pm(M,x)$ with the same assumptions as in \cite{JS}. The part in this paper where the argument would fail is in Lemma \ref{Lm:approxalpha0} as well as Proposition \ref{Pro:geom_ite}  since we could not say that if $v$ and $w$ are viscosity solutions of $F(D^2v,x)=f$, then the difference $v-w$ is in the class $S^*$ or $S$. Indeed, this is a very difficult problem in the theory of $C$-viscosity solutions.
\end{Remark}

\section{Approximation result for $C^{1,\alpha_0}$ interfaces} 
\label{Section:approxalpha0}
We can state the following approximation result, which relates our equation to the limiting profile where $g=0$ and $f^\pm=0$, for which we have $C^{1,\alpha_0}$ regularity even across the interface.

\begin{Lemma}[Approximation lemma]\label{Lm:approxalpha0}
Fix $0< \delta < 1, 0<\tau<1/4$ and let $u \in C(B_1)$ be a viscosity solution of 
\begin{align}\label{eq:mainmu}
    \begin{cases}
        |x_d-\Psi(x')|^{a(x)}F^\pm(D^2u)=f^\pm \quad &\mbox{ in } \Omega^{\pm}\\
        u^+_\nu-u^-_\nu=g &\mbox{ on } \Gamma
    \end{cases}
    \end{align}
    such that $\|u \|_{L^\infty(B_1)} \leq 1$, $[\Psi]_{C^{1,\alpha_0}(\Gamma)}\leq 1$ and
  \begin{equation}\label{eq:smallness}
\|g\|_{C^{\alpha_0}(\Gamma)}+\|f^-\|_{L^\infty(\Omega^-)}+\|f^+\|_{L^\infty(\Omega^+)}\leq \delta.
    \end{equation}
    Then there exists $v\in C^{1,\alpha_0^-}_{loc}(B_{3/4})\cap C^{0,\beta}(\overline{B_{3/4}})$ such that
    \[
    \|u-v\|_{L^\infty(B_{3/4-\tau})}\leq C(\tau^\beta+\delta),
    \]
    for  $\beta=\beta_0/2$ and for some $C>0$ depending only on $d, \lambda, \Lambda, \alpha_0$.
\end{Lemma}

\begin{proof}
We proceed very similarly to \cite[Lemma 5.2]{SS}.

Fix $0<\delta<1$, and $\theta>0$ is given by \ref{Assumption3}. Given $\varepsilon>0$ small, for $x\in B_1$ we define
\[
F_\varepsilon(M,x)=h_\varepsilon(x)F^+(M)+(1-h_\varepsilon(x))F^-(M),
\]
where  $h_\varepsilon\in C^\infty(B_1)$, $0\leq h_\varepsilon(x)\leq 1$ and
\begin{align*}
        h_\varepsilon(x)=
    \begin{cases}
    1 & \mbox{ if } x\in B_1\cap \{x_d>\Psi(x')+\varepsilon\}\\
    0 & \mbox{ if } x\in B_1\cap \{x_d<\Psi(x')-\varepsilon\}.
    \end{cases}
\end{align*}
Note that $F_\varepsilon\in \mathcal{E}(\lambda,\Lambda)$ and $F_\varepsilon(0,x)\equiv0$. Let $v_\varepsilon$ a solution to 
\[
 \begin{cases}
    F_\varepsilon(D^2 v_\varepsilon, x) =0& \mbox{in } B_{3/4}\\
    v_\varepsilon=u & \mbox{ on }\p B_{3/4}.
    \end{cases}
\]
Arguing as in the proof of \cite[Lemma 5.2]{SS}, we have that $v_\varepsilon \in C_{loc}^{1,\bar \gamma}(B_{3/4}) \cap C^{\beta}(\overline{B_{3/4}})$ for every $\bar \gamma < \alpha_0, \beta= \beta_0 /2$ with universal estimates.
Thus by compactness $v_{\varepsilon} \to v$ in $C_{loc}^{1,\gamma}(B_{3/4})\cap C^{\beta}(\overline{B_{3/4}})$ as $\varepsilon \to 0$ for every $\gamma < \bar \gamma$. By closedness of viscosity solutions under uniform limits, we also have that $v$ solves
\[
F^\pm (D^2 v)=0 \quad \mbox{ in } \Omega_{3/4}^\pm.
\]
Defining $w=u-v$, since $u \in C^{\beta}(\overline{B_{3/4}})$ from Proposition \ref{Pro:Holder_global}, then $w \in C^{\beta}(\overline{B_{3/4}})$ and $w=0$ on $\p B_{3/4}$. Thus for $0 < \tau < 1/4 $,
\[
\| w\|_{L^{\infty}(\p B_{3/4-\tau})} \leq [w]_{C^{\beta}(\overline{B_{3/4}})} \tau^\beta \leq C \tau^\beta, 
\]
with $C$ universal. Furthermore, we are able to prove that 
\[
\begin{cases}
    w \in S_{\omega}(f^\pm) & \mbox{in } \Omega_{3/4 - \tau}^\pm \\ 
    w_{\nu}^+ - w_{\nu}^- =g  & \mbox{on } \Gamma_{3/4-\tau}^\pm.
\end{cases}
\]
The second follows immediately since $v$ satisfies $v_\nu^+ - v_\nu^-=0$ in $\Gamma_{3/4}$ in the classical sense. The first condition can be checked directly: take for example $x_0 \in \Omega_{3/4}^+$ and let $r:=\frac{1}{2}\dist(x_0, \Gamma_{3/4})$,
then $|x_d-\Psi(x')|^{a(x)} \geq c > 0$ hence
\begin{align*}
    F^+(D^2 u) &= f^+ |x_d-\Psi(x')|^{-a(x)} \in L^\infty \quad \mbox{in } B_{r}(x_0) \\
    F^+(D^2 v)&= 0 \quad \mbox{in } B_{r}(x_0),
\end{align*}
which means $w \in S_{\omega}(f^+)$ in $B_{r}(x_0)$. Now, from the ABP estimate Theorem \ref{Thm:deg_ABP} and \eqref{eq:smallness} we conclude that $\| u- v \|_{L^\infty(B_{3/4- \tau})} \leq C(\tau^\beta + \delta).$
 
\end{proof}
 For $\mu > 0$ define the following operator
\begin{align}\label{eq:notation}
    F_\mu^\pm(M,x)= |x_d - \mu^{-1}\Psi (\mu x') |^{a(\mu x)} F^\pm(D^2 M).
\end{align}
Now we verify that it is possible to assume, without loss of generality, that the smallness assumption \eqref{eq:smallness} holds.
Let $u$ to be a viscosity solution of \eqref{eq:main} and consider
\[ 
\tilde u (y) = \frac{u(ry)}{r^2 K}
\]
with $$K:= r^{-2} \|u\|_{L^\infty(B_1)} + \left(\frac{\delta}{3}\right)^{-1} \left[ r^{-\bar a}(\|f^-\|_{L^\infty(\Omega^-)}+\|f^+\|_{L^\infty(\Omega^+)}) + r^{-1}   \|g\|_{C^{\alpha_0}(\Gamma)}\right]. $$
Then $\tilde u$ solves 
\begin{align*}
\begin{cases}
\displaystyle
       \frac{1}{ K} F^\pm(KD^2 \tilde u , ry )=  \frac{f^\pm(r y)}{K} r^{-a(rx)}, \quad & y \in B_{1/r}\cap\{\pm (r y_d-\Psi(r y'))>0\},  \\
      \displaystyle  \tilde u^+_\nu (y)- \tilde u^-_\nu(y)=\frac{g(r y)}{rK}, & y \in B_{1/r}\cap \{r y_d=\Psi(r y')\}.
    \end{cases}
\end{align*}
Which can be rewritten as
\begin{align*}
\begin{cases}
\displaystyle
      \tilde{F}_{r}^\pm(D^2 \tilde u , y )=  \tilde f, \quad & y \in B_{1/r}\cap\{\pm ( y_d-r^{-1}\Psi(r y'))>0\},  \\
      \displaystyle  \tilde u^+_\nu- \tilde u^-_\nu=\bar g, &y \in B_{1/r}\cap \{ y_d=r^{-1}\Psi(r y')\}.
    \end{cases}
\end{align*}
Thus, choosing $r \leq \delta$, one can check that the smallness regime is verified.
Moreover, possibly choosing $r$ smaller, if we define $\tilde \Psi (y')= \frac{\Psi(ry')}{r}$ we have
\begin{equation*}
    [\tilde \Psi]_{C^{1,\alpha_0}(0)}= \sup_{y' \in B'_1, y' \neq 0}   \frac{|D'  \tilde\Psi(y')|}{|y'|^{\alpha_0}}    = \sup_{y' \in B'_1, y' \neq 0} \frac{|D' \Psi(r y')|}{|y'|^{\alpha_0}} \leq r^{\alpha_0}   [\Psi]_{C^{1,\alpha_0}(0)} \leq 1.
\end{equation*}

\section{Gradient regularity} \label{Section:Gradient_reg}

In this section we obtain pointwise $C^{1,\alpha(x_0)}(x_0)$ regularity at points in the interface $x_0\in \Gamma$. For simplicity and without loss of generality, we assume that $x_0=0$. The next result constitutes the first geometric iteration. Note however that the approximation function $\ell$ is affine even across the interface. This is a consequence of the smallness assumption on $g$.

\begin{Lemma}\label{Lm:First_it_g0}
Let $u$ be a viscosity solution of \eqref{eq:main}, $ 0 <\gamma < \alpha_0$ and $\delta_1$ be given by in Remark \ref{Rmk:mod_cont_exp}, then there exist $\delta>0$ and $ \rho \leq \delta_1$ universal such that if \eqref{eq:smallness} holds, then there is an affine function $\ell(x)=A\cdot x + b$, with $|A|+|b|\leq C_0$, such that
\begin{align*}
\| u - \ell \|_{L^\infty(B_\rho)} &\leq \frac{\rho^{1+\gamma}}{2}. 
\end{align*} 
\end{Lemma}
\begin{proof}
 Take  $\delta >0$ to be fixed, so that if the smallness assumption \eqref{eq:smallness} is satisfied, then there exists $v \in C_{loc}^{1,\bar\gamma}(B_{3/4})$ where $\bar\gamma=(\gamma+\alpha_0)/2$, such that 
    \begin{equation*}
        \|u-v\|_{L^\infty(B_{3/4-\tau})}\leq C(\tau^\beta+\delta),
    \end{equation*}
    for every $0<\tau<1/4$.
    Now, let
    \begin{equation*}
        \ell(x) = v(0) + D v(0) \cdot x 
    \end{equation*}
    and compute 
    \begin{equation*}
        \sup_{B_{\rho}} |u(x) -  \ell(x) | \leq \sup_{B_{\rho}} |v(x) -  \ell(x) | + \sup_{B_{\rho}} |u(x) -  v(x)| \leq C \rho^{1+\bar\gamma} + C(\tau^\beta+\delta), 
    \end{equation*}
    where $C>0$ is a universal constant. Now, we choose $\rho$, $\tau$ and $\delta$ universal as
    \begin{equation*}
        \rho= \min\left\{ \left( \frac{1}{6C}\right)^{\frac{2}{\alpha_0-\gamma}}, \delta_1 \right\} \quad \tau=\left(\frac{\rho^{1+\gamma}}{6C}\right)^\frac{1}{\beta}\quad \delta=\frac{\rho^{1+\gamma}}{6C}
    \end{equation*}
    and we get
    \begin{align*}
\| u - \ell \|_{L^\infty(B_\rho)} &\leq \frac{\rho^{1+\gamma}}{2},
\end{align*} 
as intended.
\end{proof}

Now we proceed with the geometric iterations which will imply pointwise $C^{1,\alpha(\cdot)}$ regularity at points in the interface from either side, making use of Proposition \ref{Pro:new_holder}. 

\begin{Proposition}\label{Pro:geom_ite}
    Let $u$ be a viscosity solution of \eqref{eq:main}. There exist a nondecreasing sequence $(\alpha_k)_k$, universal constants $\delta$ and $\rho$ such that if the smallness assumption \eqref{eq:smallness} holds, then there are sequences of affine functions
    \[
    \ell_k^\pm(x)=\eta_k+\zeta_k^\pm\cdot x, \quad x\in \Omega_k^\pm:=B_{\rho^k}\cap\{\pm(x_d-\Psi(x'))>0\}
    \]
    such that
    \begin{align}\label{eq:geom_it}
        \|u-\ell_k^\pm\|_{L^\infty(\Omega_k^\pm)}\leq \rho^{k(1+\alpha_k)}
    \end{align}
    where 
    \begin{align*}
        \ell_k^+(x)-\ell_k^-(x)=g(0)\nu(0) \cdot x
    \end{align*}
    and
    \begin{align}\label{eq:geom_coef}
        |\eta_k-\eta_{k-1}|+\rho^k|\zeta_k^\pm-\zeta_{k-1}^\pm|\leq C\rho^{(k-1)(1+\alpha_{k-1})}.
    \end{align}
    Furthermore, the sequence $(\alpha_k)_k$ converges to 
\begin{align}\label{eq:geom_alpha}
    \alpha=\min\left\{ \alpha_0^- , 1-a(0) \right\},
\end{align}
    and
    \begin{align}\label{eq:geom_convergence}
        \limsup_{k\to \infty} k(\alpha-\alpha_k)=0.
    \end{align}
\end{Proposition}
\begin{proof}
    We argue by induction. Let $\delta$ and $\rho$ be given by Lemma \ref{Lm:First_it_g0},
    \[
    a_k:=\sup_{x\in B_{\rho^k}} a(x) 
    \]
    and define the nondecreasing sequence
    \[
    \alpha_k:=\min\left\{ \alpha_0^-, 1-a_k  \right\}
    \]
    which converges to $\alpha$ in \eqref{eq:geom_alpha}. Furthermore, by assumption \ref{Assumption2} as well as Remark \ref{Rmk:mod_cont_exp}, it is a simple calculation to see that \eqref{eq:geom_convergence} also holds. 
    Define also
    \begin{align*}
        \tilde \Omega_k^\pm&\,=B_1\cap\{\pm(x_d-\rho^{-k}\Psi(\rho^{k}x'))>0\},\\
        \tilde \Gamma_k&\,=B_1 \cap \{(x_d-\rho^{-k}\Psi(\rho^{k}x'))=0\}.
    \end{align*}

    Let $\ell_0\equiv 0$, $\ell_1$ be given by Lemma \ref{Lm:First_it_g0} and define
    \[
    \ell_1^\pm(x)=\ell_1(x)\pm \frac{g(0)}{2}\nu(0) \cdot x
    \]
    where $\nu(0)$ is the normal of $\Gamma$ at $0$  pointing towards $\Omega^+$, which we are assuming, without loss of generality, that $\nu(0)=e_d$. By Lemma \ref{Lm:First_it_g0},
    \begin{equation*}
        \|u-\ell_1^{\pm} \|_{L^{\infty}(B_{\rho})} \leq \frac{\rho^{1+\gamma}}{2} + \|g \|_{L^{\infty}(B_{\rho})}  \leq 
        \frac{\rho^{1+\gamma}}{2} + \delta \leq \rho^{1+ \gamma},
    \end{equation*}
    since we can assume $\delta \leq  \frac{\rho^{1+\gamma}}{2}$. This concludes the first step of the induction, since $\gamma>\alpha_1$ and \eqref{eq:geom_coef} can be verified directly from Lemma \ref{Lm:First_it_g0}.

    To prove the induction step,  define
    \[ 
    v_k^\pm(x) = \frac{(u - \ell_k^\pm)(\rho^k x)}{\rho^{k(1+\alpha_k)}}\quad x\in B_1,
    \]
    where $\ell_k^+(x)-\ell_k^-(x)=g(0)\nu(0) \cdot x$. The function $v_k^\pm$ solves in $\tilde \Omega_k^\pm$ the equation
\begin{align*}
         \left| x_d - \rho^{-k}\Psi(\rho^k x') \right|^{a(\rho^k x)} \rho^{k(1-\alpha_k)} F^\pm(\rho^{k(-1+\alpha_k)}D^2 v_k^\pm) \\
        =\rho^{k(1-\alpha_k-a(\rho^k x))} f^\pm(\rho^k x).
\end{align*}
    Note that $1-\alpha_k - a(\rho^k x) \geq a_k - a(\rho^k x)  \geq 0 $ from the definition of $\alpha_k$ and $a_k$. Thus, denoting with
    \[
     F_k^\pm(M) := \rho^{k(1-\alpha_k)} F^\pm(\rho^{k(-1+\alpha_k)} M)
    \]
 and recalling the notation in \eqref{eq:notation} we can write compactly the equation satisfied by $v_k$ as 
\[
(F_k^\pm)_{\rho^k}(D^2 v_k^\pm, x) = f_k^\pm  \quad \text{in }\tilde \Omega_k^\pm
\]
where  $\|f_k^\pm \| \leq \delta$.
Whereas for the transmission condition we have 
\[ 
(v_k^\pm)_\nu^+ - (v_k^\pm)_\nu^- = \frac{g(\rho^k x)- g(0)\nu(0)\cdot \nu(\rho^k x)}{\rho^{k\alpha_k}} = g_k(x)\quad x\in \tilde \Gamma_k,
\]
with 
\begin{align*}
    |g_k(x)| \leq&\, \frac{| g(\rho^k x)- g(0)|}{\rho^{k \alpha_k}} + \frac{|g(0)- g(0)\nu(0)\cdot\nu(\rho^k x)|}{\rho^{k\alpha_k}} \\
    \leq&\, [g]_{C^{\alpha_0}(\Gamma)} + |g(0)|[\Psi]_{C^{1,\alpha_0}(B'_1) }\leq \delta.
\end{align*}

Similarly, we check that $[g_k]_{C^ {\alpha_0}(\tilde \Gamma_k)}\leq \delta$.

Note however that we can not apply Lemma \ref{Lm:First_it_g0} to $v_k^\pm$ since it is discontinuous across the interface. Therefore, we proceed as in \cite[Proof of Theorem 6.1]{SS}. Indeed for $x \in \tilde \Gamma_k$, we have
\begin{align*}
    &\,|(v_k^+-v_k^-)(x)|=\frac{|\ell_k^+(\rho^k x)-\ell_k^-(\rho^k x)|}{\rho^{k(1+\alpha_k)}}=\frac{|g(0)\nu(0)\cdot x|}{\rho^{k\alpha_k}}=\frac{|g(0)x_d|}{\rho^{k\alpha_k}}\\
    \leq &\,|g(0)|\,\sup_{x\in  \tilde\Gamma_k}\frac{|x_d|}{\rho^{k\alpha_k}}=|g(0)|\,\sup_{x\in B'_1}\frac{|\Psi(\rho^k x')|}{\rho^{k(1+\alpha_k)}}\leq \|g\|_{C^{\alpha_0}( \Gamma)}\,[\Psi]_{C^{1,\alpha_0}(B'_1)}\leq \delta
\end{align*}

Consider now, $w \in C(B_1)$ the viscosity solution of
\begin{equation}
    \begin{cases}
(F_k^\pm)_{\rho^k}(D^2 w, x) = 0  &\quad \text{in }\tilde\Omega_k^\pm \\
w=\frac{1}{2}(v_k^+ + v_k^-)&\quad \text{on }\tilde\Gamma_k\\
w=v &\quad \text{on }\p B_1
    \end{cases}
\end{equation}
We will prove that $w$ satisfies the assumptions of Lemma \ref{Lm:First_it_g0}. By the maximum principle, $\|w\|_{L^\infty(B_1)}\leq \|v_k^\pm\|_{L^\infty(B_1)}\leq 1$. Furthermore, from \cite[Theorem 3]{JS} we get that $v_k^\pm-w^\pm\in S_\omega(f_k^\pm)$ in $\tilde\Omega_k^\pm$ and from the ABP in \cite[Proposition 1]{JS}, we get
\begin{align*}
    \|v_k^\pm-w^\pm\|_{L^\infty(\tilde \Omega_k^\pm)}\leq \|v_k^\pm-w^\pm\|_{L^\infty(\tilde \Gamma_k^\pm)}+C\|f_k^\pm\|_{L^d_\omega(\tilde\Omega_k^\pm)}\leq C\delta
\end{align*}
By \cite[Corollary 3.3]{BGMW}, we get pointwise boundary $C^{1,\alpha}$ estimates and thus for every $x_0\in \tilde \Gamma_k$ we have
\begin{align*}
    |D(v_k^\pm-w^\pm)(x_0)|\leq C\left( \|v_k^\pm-w^\pm\|_{L^\infty(\tilde\Omega_k^\pm)}+\frac{1}{2} \rho^{-k\alpha_k}|g(0)|\,\|\Psi_k\|_{C^{1,\alpha_0}(x_0)}+\|f_k^\pm\|_{L^p_\omega(\tilde\Omega_k^\pm)}\right),
\end{align*}
where $\Psi_k(x)=\rho^{-k}\Psi(\rho^kx)$ satisfies $\|\Psi_k\|_{C^{1,\alpha_0}(x_0)}\leq \rho^{k\alpha_0}\|\Psi\|_{C^{1,\alpha_0}(B_1')}$, with $\alpha_0>\alpha_k$ and $|g(0)|\leq \delta$.

Proceeding as in \cite{SS} we check that $w$ satisfies the assumptions of Lemma \ref{Lm:First_it_g0} and hence there exists an affine function $\bar \ell_{k+1}$ such that
\[
\|v-\bar\ell_{k+1}\|_{L^\infty(B_\rho)}\leq \|v-w\|_{L^\infty(B_\rho)}+\|w-\bar\ell_{k+1}\|_{L^\infty(B_\rho)}\leq C\delta+\frac{\rho^{1+\gamma}}{2}\leq \rho^{1+\gamma}
\]
for any $\gamma < \alpha_0$. Recalling the definition of $v_k$ and rescaling back, we get
\begin{align*}
\left\| \frac{(u -\ell_k^\pm)(\rho^k x) }{\rho^{k(1+\alpha_k)}}-  \bar \ell_{k+1}(x) \right\|_{L^{\infty}(B_\rho)} &\leq \rho^{1+\gamma}  \\ 
\left\| u(y) - \ell_k^\pm (y) -  \rho^{k(1+\alpha_k)} \bar \ell_{k+1}(\rho^{-k} y) \right\|_{L^{\infty}(B_{\rho^{k+1}})} &\leq \rho^{1+\gamma+k(1+\alpha_k)}
\end{align*}
 hence defining $\ell_{k+1}^\pm(y) = \ell_k^\pm (y) +  \rho^{k(1+\alpha_k)} \bar \ell_{k+1}(\rho^{-k} y) $ we finally obtain
\begin{align*}
\| u - \ell_{k+1}^\pm \|_{L^\infty(B_{\rho^{k+1}})} &\leq \rho^{1+\gamma+k(1+\alpha_k)} \leq \rho^{(k+1)(1+\alpha_{k+1})}
\end{align*}
where in the last inequality we have exploited the monotonicity of $\alpha_k$ and the fact that we can choose $\gamma \geq \alpha_{k+1}$. This concludes the induction step for \eqref{eq:geom_it}. 

Finally, \eqref{eq:geom_coef} holds immediately from the bounds on the coefficients of $\bar \ell_{k+1}$ obtained in Lemma \ref{Lm:First_it_g0}.
\end{proof}

In order to get the full regularity up to the interface of Theorem \ref{Thm:main} we patch the pointwise regularity at the interface of Proposition \ref{Pro:geom_ite} and the interior $C^{1,\alpha(\cdot)}$ regularity, since $\alpha(\cdot) < \alpha_0$, proceeding as in \cite{JS} with the obvious changes, see also \cite[Proposition 2.3]{MS}.

\section{Acknowledgements}
The authors would like to thank Diego Moreira, Yannick Sire and María Soria-Carro for fruitful
conversations on the topic of this paper.

\end{document}